\documentclass[12pt,reqno]{amsart}
\usepackage{graphicx}
\usepackage{amsmath,amssymb,amsthm}
\usepackage{enumitem}
\usepackage[hmargin=1in,vmargin=1in]{geometry}
\usepackage[unicode,breaklinks=true,colorlinks=true]{hyperref}
\usepackage[dvipsnames]{xcolor}
\usepackage{color}
\usepackage{cite}

\usepackage[normalem]{ulem} 
\usepackage{comment}
\usepackage{tikz}
\usepackage{psfrag}
\usetikzlibrary{patterns}
\usepackage{float}
\usepackage{epsfig} 

\newtheorem{thm}{Theorem}[section]
\newtheorem{lem}[thm]{Lemma}
\newtheorem{prop}[thm]{Proposition}

\newtheorem{defn}[thm]{Definition}

\theoremstyle{definition}
\newtheorem{eg}[thm]{Example}
\newtheorem{obs}[thm]{Remark}

\newcommand{\tnum}{\rm(\roman*)}
\newcommand{\rnum}{\rm(\alph*)}
\newcommand{\nnum}{\rm(\arabic*)}

\newcommand{\ord}{{\rm{ord}_{\mathcal{S}}}}

\newcommand{\succup}{%
      \rotatebox{90}{$\succ$}}
\newcommand{\succdn}{%
      \rotatebox{-90}{$\succ$}}

\usepackage{graphicx,psfrag}
\numberwithin{equation}{section}

\def\k0{\kappa_0}

\def\tA{\tilde{A}}
\def\tB{\tilde{B}}
\def\tu{\tilde{u}}
\def\tv{\tilde{v}}
\def\tf{\tilde{f}}
\def\tg{\tilde{g}}

\def\tbx{\tilde{\bx}}
\def\tt{\tilde{t}}

\def\bu{{\bf{u}}}

\def\bB{{\bf{B}}}
\def\bF{{\bf{f}}}
\def\bg{{\bf{g}}}
\def\bk{{\bf{k}}}
\def\bj{{\bf{j}}}
\def\bx{{\bf{x}}}

\def\b0{{\bf{0}}}
\def\Bs{B_{\rm s}}

\begin{document}
\title[NSE Galerkin with Large Grashof Numbers]
{On Galerkin approximations of the Navier--Stokes equations in the limit of large Grashof numbers}
\author[C. Foias]{Ciprian Foias}
\author[L. Hoang]{Luan Hoang$^{1}$}
\address{$^1$Department of Mathematics and Statistics,
Texas Tech University\\
1108 Memorial Circle, Lubbock, TX 79409--1042, U. S. A.}
\email{luan.hoang@ttu.edu}

\author[M. S. Jolly]{Michael S. Jolly$^2$}
\address{$^2$Department of Mathematics\\
Indiana University\\ Bloomington, IN 47405, U. S. A.}
\email{msjolly@indiana.edu}


\date{\today}

\begin{abstract}
We examine how stationary solutions to Galerkin approximations of the Navier--Stokes equations behave in the limit as the Grashof number $G$ tends to $\infty$.  An appropriate scaling is used to place the Grashof number as a new coefficient of the nonlinear term, while the body force is fixed. A new type of asymptotic expansion, as $G\to\infty$, for a family of solutions is introduced. Relations among the terms in the expansion are obtained by following a procedure that compares and totally orders positive sequences generated by the expansion. The same methodology applies to the case of perturbed body forces and similar results are obtained. We demonstrate with a class of forces and solutions that have convergent asymptotic expansions in $G$. All the results hold in both two and three dimensions, as well as for both no-slip and periodic boundary conditions.
\end{abstract}

\maketitle

\tableofcontents


\section{Introduction} 

 Turbulent behavior of solutions to the Navier--Stokes equations (NSE) is generally expected for very large Grashof numbers $G$ (see definition \eqref{Grashof}) \cite{F97}.  This is preceded by progressively complicated behavior typically observed as $G$ increases (for an exception see \cite{Marchioro}).  In general, the unique, asymptotically stable steady state at very small values of $G$ bifurcates.   Depending on the particular body force, this produces multiple steady states and/or periodic solutions which in turn, often period double, breaking symmetries and leading to dynamics of increasing complexity \cite{Frisch}.  That long time behavior is contained in the global attractor $\mathcal{A}$ whose fractal dimension, $\text{dim}_{\text{F}}(\mathcal A)$, is bounded by $cG$, ($cG^{2/3}(\log G)^{1/3}$ for periodic boundary conditions) for some constant $c$ \cite{T97}.

The Galerkin approximation method plays an important role in the study of the NSE.  It is used both to prove existence of solutions and as an approach to numerical simulations. 
As $G$ increases, one can expect to need more terms in the Galerkin basis (at least comparable to $\text{dim}_{\text{F}}(\mathcal A)$) to obtain accurate numerical results.  Nevertheless, in practice, one often seeks insight in the behavior of solutions for large $G$ for a fixed Galerkin approximation.  

In this paper we study the limit as $G_n\to \infty$ for sequences of steady state solutions $(v_n)_{n=1}^\infty$ of a fixed Galerkin approximation of the NSE. The results apply to both periodic and no-slip boundary conditions in both two and three dimensions. (For certain properties of the set of steady state solutions of the NSE, see \cite{FT76,FT77,FT78,FSequil83,SaTe80}.) This is done through a new type of expansion in the form
\begin{equation*}
v_n=v+\Gamma_{1,n}w_1+\Gamma_{2,n}w_2+\cdots+
\Gamma_{k,n}w_k+\cdots,
\end{equation*}
where $\Gamma_{k,n}$'s are positive numbers converging to zero as $n\to\infty$ and $w_k$'s are unit vectors.
We show that any convergent sequence has a subsequence with such an expansion.
The payoff is a collection of results that prescribe particular algebraic relations between the limit, $v$, and some of the fixed vectors $w_k$ in the expansion.  Since no assumptions are made on the body force, there are a multiple, yet finite number of possible cases, independent of the number of modes in the Galerkin approximation.  Which case applies for a particular force can be suggested by comparing combinations of $\Gamma_{k,n}$ with $G_n$, as we demonstrate with a numerical simulation.

We now describe the progression of the paper.
In Section \ref{prelim}, after recalling basic facts about the NSE, we rescale it, see \eqref{changevar}, \eqref{Grashof} and \eqref{subs}, to equation \eqref{sc} with the body force having a fixed $L^2$-norm.
We state the main results in Section \ref{results}. Their proofs are left to Sections \ref{rigor} and \ref{mainproofs}. The starting point is the asymptotic expansion \eqref{infser} which leads to the formal equation \eqref{sssub}.
The idea is to order the sequences of coefficients in \eqref{sssub}, as $n\to\infty$. Such an order is introduced in Definition \ref{deford}. The rigorous treatment requires Lemma \ref{total}, Proposition \ref{Sprop} and a set of equivalence classes 
which finally  lead to Definition \ref{ordef}. The main idea is summarized in the General Procedure \eqref{GP}. Our main results concerning the asymptotic expansion \eqref{infser} for the NSE are Theorem \ref{thm1} and Propositions \ref{v0eg}, \ref{vAinvg} in subsection \ref{subNSE}. 

To make rigorous the comparison of terms in different sequences requires some general results, independent of the NSE. 
In Section \ref{rigor}, a new notion of asymptotic expansions - called strict unitary expansions - is defined in Definition \ref{uexp}. It, in fact, is inspired by Lemma \ref{mainlem} which proves that any bounded sequence in a finite dimensional normed space has a subsequence that possesses such a strict unitary expansion. The proof of that lemma involves taking convergent subsequences on the unit sphere, which is where the finite dimensionality of  Galerkin approximations is needed. Numerical results are presented in subsection \ref{subnum} for the 3D, periodic case, which demonstrate both the construction of the expansions as well as a particular case of Theorem \ref{thm1}.
The strict unitary expansions turn out to be ``asymptotically unique" as stated and proved in Proposition \ref{unique}. 
Subsection \ref{sueeg} contains some examples which show that the strict unitary expansion can be very simple in $\mathbb R$, but can also be quite different from the Taylor expansion for the same sequence. 
With these facts at hand, we present in Section \ref{mainproofs} the 
proofs of the main results that were stated in Section \ref{results} .  

We then apply these ideas to the case of a perturbed force, $g_n$,
in Section \ref{expandsec}. The main result is Theorem \ref{gnthm}, which relates expansions of both $g_n$ and the solution $v_n$ in \eqref{genevg} with the Grashof number $G_n$. Theorem \ref{nonzthm} and Remark \ref{matrmk} show that in certain cases the procedure \eqref{GP} in Section \ref{results} works for equation \eqref{eqforce}. Theorem \ref{vzthm2} and Remark \ref{gcompare} show that a different asymptotic expansion of $g_n$ may lead to an unexpected expansion of the solution $v_n$.
This section may be useful in interpreting numerical results or detecting when they are false.
The reason is that the solutions solved numerically for the case of a fixed force $g$ are, inevitably, inexact.   Yet they can be viewed as satisfying the Galerkin approximation with a different force $g_n$, close to $g$.
Some of the more technical proofs are gathered in the Appendix.

Finally, it is worth mentioning that the new type of asymptotic expansion in this paper can be applied to other nonlinear ordinary/partial differential equations. In particular, the Kuramoto--Sivashinsky, Rayleigh--B\'enard, and magnetohydrodynamic equations can all be expressed in the same the functional form \eqref{NSE}, for certain linear operators $A$ and bilinear operators $B$ \cite{T97}.

\section{Preliminaries}\label{prelim}

The incompressible NSE for velocity $\bu(\bx,t)$, pressure $p(\bx,t)$ with a time-independent body force $\bF(\bx)$, for the spatial variables $\bx\in\mathbb R^d$ ($d=2,3$) and time variable $t\ge 0$, are
\begin{equation}\label{NSE}
\left\{ 
\begin{aligned}
&\frac{\partial \bu}{\partial t} -
\nu \Delta \bu  + (\bu\cdot\nabla)\bu + \nabla p = \bF , \\
&\text {div} \bu = 0,
\end{aligned}
\right.
\end{equation}
where $\nu>0$ is the kinematic viscosity.

\medskip
\noindent
\textbf{Periodic boundary condition.} We consider the NSE with periodic boundary conditions in $\Omega =[0,L]^d$ for some $L>0$.
Let $\mathcal V$ denote the set of $\mathbb R^d$-valued $\Omega$-periodic divergence-free trigonometric polynomials with zero average over $\Omega$.
Define the spaces $H=$closure of $\mathcal V$ in $L^2(\Omega)^d$, $V=$closure of $\mathcal V$ in $H^1(\Omega)^d$. 

\medskip
\noindent\textbf{No-slip boundary condition.} 
Consider the NSE in an open, connected, bounded set $\Omega$ in $\mathbb R^d$ with $C^2$-boundary $\partial \Omega$. Let $\mathbf n$ denote the outward normal vector on the boundary. 
Under the no-slip boundary condition 
$$\bu =0\text{ on }\partial \Omega,$$
the relevant spaces are
\begin{align}
H=\{\bu\in L^2(\Omega)^d:\ \nabla \cdot \bu=0, \quad \bu\cdot\mathbf n|_{\partial \Omega}=0\},\
V=\{\bu\in H_0^1(\Omega)^d:\ \nabla \cdot \bu=0 \}.
\end{align}

\medskip
The  scalar product and its associated norm in $H$ are those of $L^2(\Omega)^d$ and are denoted by $\langle\cdot,\cdot\rangle$ and $|\cdot|$, respectively.
(We also use $|\cdot|$ for the modulus of a vector in 
$\mathbb{C}^d$; we assume that the 
meaning will be clear from the context.)

Let $\mathcal{P}$ be the Helmholtz-Leray projection, 
that is, the orthogonal projection in $L^2(\Omega)^d$ onto $H$. 
The Stokes operator is $A=-\mathcal P\Delta$ defined on  $D(A):=V\cap H^2(\Omega)^d$ which is  called the domain of $A$. The domain of $A^{1/2}$ is  $D(A^{1/2})=V$, and the natural 
norm on $V$ is
\[
\|u\|=|A^{1/2}u|=\left(\int_{\Omega}\sum_{j=1}^{d} 
\frac{\partial}{\partial x_j}u(\bx)\cdot 
\frac{\partial}{\partial x_j}u(\bx)  d\bx\right)^{1/2}.
\]
The bilinear operator $B$ is defined as 
\[
B(u,v)=\mathcal{P}\left( (u \cdot \nabla) v \right),\text{ for } u,v\in D(A).
\]

Assume $\bF\in H$.  The NSE \eqref{NSE} can be written as
a differential equation in an appropriate functional  space 
(see \cite{CF88} or \cite{T97}), 
\begin{equation}\label{fNSE}
\frac{du}{dt} + \nu Au + B(u,u) = f ,
\end{equation}
where $u(t)=\bu(\cdot,t)$ and $f=\bF(\cdot)$.  

The operator $A$ is positive, self-adjoint with a compact inverse, 
and its eigenvalues are positive numbers
$$\lambda_1\le \lambda_2\le \lambda_3\le \ldots \text{ satisfying } \lim_{n\to\infty}\lambda_n=\infty.$$
Moreover, each eigenspace of $A$ is finite dimensional.
The bilinear term enjoys the   orthogonality relations 
(see for instance \cite{T97})
\begin{equation}\label{Bflip}
\langle B(u,v),w\rangle = -\langle B(u,w),v\rangle,\quad  \text{so that} \quad 
\langle B(u,v),v\rangle=0\text{ for } u,v,w\in D(A).
\end{equation}

We denote 
$\kappa_0=\lambda_1^{1/2}$
and make the following  change of variables
\begin{equation}\label{changevar}
\begin{aligned}
\bu(\bx,t)&=\nu\k0\tilde \bu(\k0 \bx,\nu\k0^2 t)=\nu\k0\tilde\bu(\tbx,\tt),\\
p(\bx,t)&=\nu^2\k0^2\tilde p(\k0 \bx,\nu\k0^2 t)=\nu^2\k0^2\tilde p(\tbx,\tt),\\
\bF(\bx)&=\nu^2\k0^3\tilde \bF(\k0 \bx)=\nu^2\k0^3\tilde\bF(\tbx),
\end{aligned}
\end{equation}
to obtain
\begin{equation}\label{changesys}
\left\{ 
\begin{aligned}
&\frac{\partial \tilde\bu}{\partial \tilde t} - \Delta_{\tbx} \tilde\bu  + (\tilde\bu\cdot\nabla_{\tbx})\tilde\bu + \nabla_{\tbx}\tilde p = \tilde\bF,\\
&\text {div}_{\tbx} \tilde\bu = 0,
\end{aligned}
\right.
\end{equation}
for the new domain $\tilde \Omega=\k0 \Omega$. One can verify that the corresponding boundary conditions for $\tilde \Omega$ are also satisfied.
Similar to \eqref{fNSE} we can rewrite \eqref{changesys} as 
\begin{equation}\label{cov}
\frac{d\tu}{d\tt}+ \tA\tu+ \tB(\tu,\tu)=\tf,
\end{equation}  
where $\tilde u(\tt)=\tilde \bu(\cdot,\tt)$ and $ \tf=\tilde \bF(\cdot)$.

The corresponding Stokes operator is $\tilde A$ which has positive, strictly increasing eigenvalues $\tilde\lambda_n$ that go to infinity as $n\to\infty$.

For the periodic boundary condition, $A=-\Delta $ on $D(A)$ and the eigenvalues of $A$ are of the form
\[
\left(\frac{2\pi}{L}\right)^2\bk\cdot \bk, 
\quad \mbox{where} \ \bk \in \mathbb{Z}^d \setminus \{ 0\}.
\]
Therefore, $\lambda_1=(2\pi /L)^2$, $\kappa_0=2\pi/L$, $\tilde \Omega=[0,2\pi]^d$ and $\tilde\lambda_1=1$.

For the no-slip boundary condition, if $\lambda$ is engeinvalue of $A$ then\begin{equation}\label{eigen1}
     - \Delta \bu  =\lambda \bu - \nabla q , \quad  \text {div} \bu = 0,\quad \bu|_{\partial\Omega}=0, \text{ for }\bu=\bu(\bx)\ne 0.
\end{equation}
Similar to \eqref{changevar}, we make the following change of variables $\tbx=\k0 \bx$, $\bu(\bx)=\tilde\bu(\tbx)$ and $q(\bx)=\k0\tilde q(\tbx)$.
Then the eigenvalue problem \eqref{eigen1} is equivalent to $$ -  \Delta_{\tbx} \tilde \bu  =\lambda \kappa_0^{-2} \tilde \bu -  \nabla_{\tbx} \tilde q , \quad  \text {div}_{\tbx} \tilde \bu = 0, \quad \tilde\bu|_{\partial\tilde\Omega}=0, \text{ for }\tilde \bu\ne0,$$
which is the eigenvalue problem for $\tilde A$.
Thus $\tilde \lambda_1=\lambda_1\kappa_0^{-2}=1$.

\medskip
The dimensionless Grashof number is given by
\begin{equation}\label{Grashof}
G=\|\tf\|_{L^2(\tilde\Omega)^d}
=\left(\int_{\tilde\Omega} |\tilde\bF(\tbx)|^2 d\tbx\right)^{1/2}
=\frac{1}{\nu^2\k0^{3-d/2}}\left(\int_{\Omega} |\bF(\bx)|^2 d \bx\right)^{1/2}.
\end{equation}

We set
\begin{equation}\label{subs}
\tg=\tf/G\quad \text{and} \quad \tv=\tu/G. 
\end{equation}
Dividing \eqref{cov} by $G$, using the bilinearity of  
$\tB$, and suppressing the tildes, we have 
\begin{equation}\label{sc}
\frac{dv}{dt}+Av+ GB(v,v) = g , \quad  v \in  H.
\end{equation}  

Note that $g$ in \eqref{sc} satisfies $\|g\|_{L^2(\Omega)^d}=1$. Also, the Stokes operator $A$ in \eqref{sc} has the first eigenvalue $\lambda_1 = 1$, which implies
\begin{equation}\label{AVuineq}
    |Au|\ge \|u\|\ge |u| \text{ for all }u\in D(A).
\end{equation}

\medskip
\noindent\textbf{The Galerkin approximation.} Let $(\varphi_k)_{k=1}^\infty$ be an orthonormal basis of $H$ for which each $\varphi_k$ is an eigenfunction of $A$. (Such a basis exists thanks to the standard theory of the Stokes operator.)
Define $\Pi_N$ to be the orthogonal projection in $H$ onto the subspace span$\{\varphi_1,\varphi_2,\ldots,\varphi_N\}$.
The Galerkin 
approximation of \eqref{sc} is 
\begin{equation}\label{rescaled}
\frac{dv}{dt}+A v+ G\Pi_N B(v,v) = \Pi_N g , \quad  v \in
\Pi_N H.
\end{equation}

Let $\mathcal H=\Pi_N H$. Then $\mathcal H$ is a finite dimensional Hilbert space with the scalar  product and norm inherited from $H$. 
The corresponding steady state equation  for \eqref{rescaled}, after suppressing 
projector $\Pi_N$, is
\begin{equation}\label{steadyG}
Av+ G B(v,v)=g,
\end{equation}
where $g,v\in\mathcal H$, and we re-denoted
\begin{equation}\label{Bbar}
B(u,v)=\Pi_N \mathcal P ((u\cdot\nabla) v)\text{ for }u,v\in \mathcal H.    
\end{equation} 

We will examine the behavior of solutions of \eqref{steadyG} for 
Grashof number $G \to \infty$.

In the presentation below, for convenience, we denote
$$\Bs(u,v)=B(u,v)+B(v,u).$$

Similar to \eqref{Bflip}, we have for $B$ defined in \eqref{Bbar} that
\begin{equation}\label{flip}
\langle B(u,v),w\rangle = -\langle B(u,w),v\rangle,\quad  \text{and} \quad 
\langle B(u,v),v\rangle=0\text{ for } u,v,w\in \mathcal H.
\end{equation}

Very often, $\Pi_N=\bar P_\Lambda$ for some eigenvalue $\Lambda$ of $A$. 
Here, even for any number $\Lambda\ge \lambda_1$, $\bar P_\Lambda$ is  the orthogonal projection in $H$ onto the sum of all eigenspaces of $A$ corresponding to the eigenvalues that are less than or equal to $\Lambda$.

For the periodic boundary condition, we may express an element in $H$ as a Fourier series
\begin{equation*}
u=\sum_{\bk\in \mathbb{Z}^d}\hat \bu_{\bk} e^{i \bk\cdot \bx},
\text{ where } 
\hat \bu_{\bk}\in\mathbb C^d,\ 
\hat \bu_0=0,\ 
\overline{\hat \bu_\bk}=\hat \bu_{-\bk},\ 
\bk \cdot \hat \bu_\bk=0,\ 
\sum_{\bk\in \mathbb{Z}^d}|\hat \bu_{\bk}|^2 <\infty. 
\end{equation*}
Then we explicitly have$$\bar P_\Lambda u=\sum_{\bk\in \mathbb{Z}^d,|\bk|^2\le \Lambda}\hat u_{\bk} e^{i \bk\cdot \bx}.$$

Denote by $\{\mathbf e_1,\mathbf e_2,\mathbf e_3\}$ the standard canonical basis of $\mathbb C^3$ and $\mathbb R^3$.

\section{Main results} \label{results}

Assume that the Grashof number $G$ takes the positive values $\alpha_n$, for $n\ge 1$, with $\alpha_n\to\infty$ as $n\to\infty$. For $n\ge 1$, let $v_n$ be a solution of  the corresponding steady state equation \eqref{steadyG}, i.e., 
\begin{equation}\label{steady}
Av_n+ \alpha_n B(v_n,v_n)=g.
\end{equation}
Throughout this section, the function $g$ is a given, fixed element in $\mathcal H\setminus\{0\}$.
In  general, the $H$-norm $|g|$ needs not be $1$, and hence, in that case, the real Grashof number is $\alpha_n|g|$.

Utilizing \eqref{AVuineq} and \eqref{flip}, we derive from \eqref{steady} that $|v_n|\le |g|$.
Hence, $(v_n)_{n=1}^\infty$ is a bounded sequence in the finite dimensional space $\mathcal H$.
We consider a finite or infinite asymptotic expansion, as $n\to\infty$, of the sequence 
$(v_n)_{n=1}^\infty$ of the form
\begin{equation}\label{infser}
v_n=v+\Gamma_{1,n}w_1+\Gamma_{2,n}w_2+\cdots+
\Gamma_{k,n}w_k+\cdots
\end{equation}
where, for all $k$
\begin{equation}\label{b11}
\lim_{n\to\infty} \Gamma_{k,n}=0,\quad     \lim_{n\to\infty} \frac{\Gamma_{k+1,n}}{\Gamma_{k,n}}=0, \quad \text{and} \quad |w_k| =1 . 
\end{equation}

In Section \ref{rigor}, such an expansion \eqref{infser} is defined precisely in Definition \ref{uexp} and is proved to hold for a subsequence of $v_n$, which we can still denote by $v_n$ itself (see Lemma \ref{mainlem}).
Substituting the expansion \eqref{infser} of $v_n$  into equation \eqref{steady}, one formally has
\begin{equation}\label{sssub}
\begin{aligned}
&\alpha_n B(v,v)+ (Av-g) +\Gamma_{1,n}Aw_1+ \cdots +
\Gamma_{k,n}Aw_k + \cdots \\
&+\alpha_n\Gamma_{1,n}\Bs(v,w_1) + \cdots +
\alpha_n\Gamma_{k,n}\Bs(v,w_k) + \cdots \\
&+\alpha_n\Gamma_{1,n}\Gamma_{1,n}B(w_1,w_1)
+ \cdots + \sum_{j=1}^k
  \alpha_n\Gamma_{j,n}\Gamma_{k-j,n}
B(w_j,w_{k-j}) + \cdots = 0 .
\end{aligned}
\end{equation}
Considering the sequences of coefficients in \eqref{sssub}, we denote
\begin{equation}\label{coeffdef}
\begin{array}{lll}
\sigma_0=(1)_{n=1}^{\infty},&
\sigma_k=(\Gamma_{k,n})_{n=1}^{\infty}&\text{for } k=1,2, \dots,
\\
\sigma_{0,0}= (\alpha_{n})_{n=1}^{\infty},&
\sigma_{0,k}= (\alpha_{n}\Gamma_{k,n})_{n=1}^{\infty} 
 & \text{for } k=1,2, \dots,\\
& \sigma_{j,k}= (\alpha_{n}\Gamma_{j,n}\Gamma_{k,n})_{n=1}^{\infty} &
 \text{for } j=1,2, \dots, \ k \ge j.
\end{array}
\end{equation}

The next idea is to compare the sequences in \eqref{coeffdef} as $n\to\infty$.

\begin{defn} \label{deford} 
Let $\mathcal X$ be the collection of all sequences of positive numbers. 
Given two sequences $\xi=(\xi_n)_{n=1}^{\infty}$, 
$\eta=(\eta_n)_{n=1}^{\infty}$ in $\mathcal X$, we
write  $\xi \succ\eta$, if $\xi_n/\eta_n \to \infty$ as $n \to
\infty$, and  $\xi \sim \eta$ if  $\xi_n/\eta_n \to \lambda$, for some
$\lambda \in (0,\infty)$.  We write $\xi\succsim \eta$ if either  $\xi\succ\eta$ or $\xi \sim\eta$. 

A subset $X$ of $\mathcal X$ is called \emph{totally comparable} if it holds for any $\xi,\eta\in X$ that $\xi\sim \eta$ or $\xi\succ \eta$ or $\eta\succ \xi$.
\end{defn}
Clearly, the relation $\sim $ in Definition \ref{deford} is an equivalence relation on $\mathcal X$ but $\succsim$ is not an order on $\mathcal X$.

Since $\alpha_n \to \infty$, $\Gamma_{k,n} \to
0$ as $n \to \infty$, and $(\Gamma_{k,n})_{n=1}^\infty\succ (\Gamma_{k+1,n})_{n=1}^\infty$ we have the following relations between 
the sequences in \eqref{coeffdef}
\begin{equation}\label{seqarray}
\begin{array}{cccccccccc}
\sigma_0    &\succ &\sigma_1 &\succ &\sigma_2 &\succ \cdots  &\succ
 &\sigma_k &\succ \cdots  \\
\succup & & \succup & & \succup & & & \succup\\
\sigma_{0,0}&\succ &\sigma_{0,1} &\succ &\sigma_{0,2} &\succ \cdots  &\succ
&\sigma_{0,k} &\succ \cdots \\
& & \succdn & & \succdn & & & \succdn \\
            & &    \sigma_{1,1} &\succ &\sigma_{1,2} &\succ \cdots &\succ
&\sigma_{1,k} &\succ \cdots \\
& &  & & \succdn & & & \succdn \\
 &  & & & \sigma_{2,2} &\succ \cdots  &\succ &\sigma_{2,k} &\succ \cdots \\
& &  & &  & & & \succdn \\
& & & & & \ddots   & & \vdots \\
& &  & &  & & & \succdn \\
 &  & & & &  & & \sigma_{k,k} &\succ \cdots 
\end{array}
\end{equation}

Let $\mathcal S$ denote the set of sequences in \eqref{coeffdef}. 
Then the set $\mathcal S$ may not be totally comparable.
However, by going through a diagonal process extracting successive subsequences, we may assume $\mathcal S$ actually is.
We express this rigorously in the following.

Let $(\varphi(n))_{n=1}^\infty$ be a subsequence of of $(n)_{n=1}^\infty$.
For $X\subset \mathcal X$, define
$$X_\varphi=\left \{ (x_{\varphi(n)})_{n=1}^\infty \text{ with } (x_n)_{n=1}^\infty \in X\right \}.$$ 
We call $X_\varphi$ a \emph{subsequential set} of $X$.
Note, for $(x_n)_{n=1}^\infty,(y_n)_{n=1}^\infty\in \mathcal X$, that if 
$(x_n)_{n=1}^\infty\succ  (y_n)_{n=1}^\infty$, respectively, $(x_n)_{n=1}^\infty\sim (y_n)_{n=1}^\infty$, then, clearly, 
$$(x_{\varphi(n)})_{n=1}^\infty\succ (y_{\varphi(n)})_{n=1}^\infty,\text{ respectively, }
(x_{\varphi(n)})_{n=1}^\infty \sim (y_{\varphi(n)})_{n=1}^\infty.$$
Thus, if $X$ is totally comparable, then so is the set $X_\varphi$.

We prove the following basic fact in the Appendix.
\begin{lem}\label{total}
If $X$ is a countable subset of $\mathcal X$, then it has a totally comparable subsequential set.
\end{lem}

By Lemma \ref{total} and the fact  $\mathcal S$ is countable, we can assume, by using a subsequential set, that $\mathcal S$ is totally comparable. 
Let $\widehat{\mathcal S}$ be the collection of equivalence classes of elements in $\mathcal S$ with respect to the relation $\sim$ in Definition \ref{deford}.
For each $\sigma\in \mathcal S$, we use the standard notation $\langle \sigma\rangle$ for the equivalent class of $\sigma$.
Define on $\widehat{\mathcal S}$ the relation $<$ by 
\begin{equation}
    \forall \sigma,\sigma'\in \mathcal S: \langle \sigma\rangle < \langle\sigma'\rangle \text { if and only if }\sigma\succ \sigma'.
\end{equation}
Then the relation $\le $ is an order on $\widehat{\mathcal S}$.
Moreover, $(\widehat{\mathcal S},\le )$ is totally ordered. 

We obtain the following important properties which will be proved in Section \ref{mainproofs}.

\begin{prop}\label{Sprop}
The following statements hold.
\begin{enumerate}[label=\tnum]
\item The cardinality of each equivalence class in $\widehat{\mathcal S}$ is finite.
\item The set $\widehat{\mathcal S}$ is well-ordered with respect to $\le$.
\end{enumerate}
\end{prop}

Because $\widehat{\mathcal S}$ is well-ordered, it is order-isomorphic to an ordinal. 
This fact prompts the following key definition.
\begin{defn}\label{ordef}
    We associate to each $x\in \widehat{\mathcal S}$ a nonzero ordinal number
${\rm ord}_{\widehat{\mathcal S}}(x)$ such that
\begin{align*}
x < y \text{ if and only if }  {\rm ord}_{\widehat{\mathcal S}}(x) <{\rm ord}_{\widehat{\mathcal S}}(y)\text{ for any } x,y\in\widehat{\mathcal S},
\end{align*}
and $\min\{{\rm ord}_{\widehat{\mathcal S}}(x): x\in \widehat{\mathcal S}\}=1$.

Define the function $\ord$ from $\mathcal S$ to the set of ordinal numbers by
 $\ord(\sigma)={\rm ord}_{\widehat{\mathcal S}}(\langle \sigma\rangle)$ for any $\sigma\in \mathcal S$.
\end{defn}

Then one has, for any $\sigma,\sigma'\in \mathcal S$,
\begin{align*}
\sigma \succ \sigma' &\text{ if and only if }  \ord(\sigma) <
\ord(\sigma'), \\
\sigma \sim \sigma' &\text{ if and only if }  \ord(\sigma) =
\ord(\sigma'), 
\end{align*}
and 
$$\min\{\ord(\sigma):\sigma\in \mathcal S\}=1.$$
Consequently, for each $\sigma_* \in \mathcal{S}$, we have equality of the sets
\begin{equation}\label{orule}
\{\ord(\sigma):\sigma\in \mathcal S,\sigma\succ \sigma_*\}
=\{\text{ordinal number } \zeta: 1\le \zeta<\ord(\sigma_*)\}.
\end{equation}

Based on Proposition \ref{Sprop} and Definition \ref{ordef}, we can naturally think of the following procedure. 

\medskip
\noindent\textbf{General Procedure} 
\begin{align}\label{GP}
&\text{Given $x\in \ord(\mathcal S)$, use \eqref{sssub} to find an equation with finitely} \\
&\text{many terms containing the sequences $\sigma\in \mathcal S$ with $\ord(\sigma)=x$.}
\end{align}

\medskip
This procedure may involve ordinal numbers which are not finite.
Denote by $\omega$ the first transfinite number beyond the positive integers.
Recall that 
$$\omega\cdot n=\underbrace{\omega+  \cdots+ \omega}_{n
  \text{ terms}}\text{ and }   \omega^2=\omega\cdot\omega =
  \omega + \omega + \omega + \cdots$$

\subsection{Implications for the NSE}\label{subNSE}
Throughout this subsection, we assume $(v_n)_{n=1}^\infty$, satisfying \eqref{steady}, with $v_n \to v$  has a strict unitary expansion (see Definition \ref{uexp}).  This amounts to  $(v_n)_{n=1}^\infty$ satisfying \eqref{infser} and \eqref{b11} with additional properties.  The proofs of each result in this subsection are in Section \ref{mainproofs}.

We start with the following result on the order of the sequences in \eqref{seqarray} and how the order imposes certain algebraic relations on the limit $v$ and expansion vector $w_1$. 

\begin{thm}\label{thm1} 
We have the following.
   \begin{enumerate}[label=\tnum]
       \item\label{p0} $\ord(\sigma_{0,0})=1$, $\ord(\sigma)>1$ for any $\sigma\in \mathcal S\setminus\{\sigma_{0,0}\}$, and 
        \begin{equation}\label{Bzero}
        B(v,v)=0.
        \end{equation}

       \item\label{p1} If the expansion \eqref{infser} is trivial, i.e., $v_n=v$ for all $n$, then 
       \begin{equation}\label{Avg}
           Av=g.
       \end{equation}
       
       \item\label{p2} If the expansion \eqref{infser} is nontrivial, then comparing $\sigma_0$ and $\sigma_{0,1}$ gives the following.
   \begin{enumerate}[label=\rnum]
       \item If $\sigma_0\succ \sigma_{0,1}$, then $\ord(\sigma_0)=2$, $\ord(\sigma_{0,1})=3$ and
       \eqref{Avg} holds.

       \item If $\sigma_{0,1}\succ \sigma_0$, then $\ord(\sigma_{0,1})=2$,
       \begin{equation}\label{os2}
           \ord(\sigma_0),\ord(\sigma_{1,1})\geq 3, \quad \ord(\sigma_1)\ge 4,
       \end{equation}
       and 
       \begin{equation}\label{rel1}
           \Bs(v,w_1)=0.
       \end{equation}
       
       \item If $\sigma_{0,1}\sim \sigma_0$, then $\ord(\sigma_0)=\ord(\sigma_{0,1})=2$,
       \begin{equation}\label{os1}
           \ord(\sigma_1),\ord(\sigma_{1,1})\ge 3,
       \end{equation}
       and
       $$Av+\lambda\Bs(v,w_1)=g, \text{ for some constant }\lambda> 0.$$
   \end{enumerate}
   
   \item\label{p3} For all $k\ge 1$,
   \begin{equation}\label{zerok}
       k< \ord(\sigma_{0,k}) \le k(k+3)/2+1.
   \end{equation}
   
    \item\label{p4} For all $1\le j \le k$, 
    \begin{equation}\label{ojk}
    k+j< \ord(\sigma_{j,k}) \le \omega\cdot (2j) + (k-j)(k-j+1)/2.    
    \end{equation}

  \item\label{p5} For all $k\ge 0$, 
    \begin{equation}\label{ok}
    k+1< \ord(\sigma_k)\le \omega^2+k.
    \end{equation}
\end{enumerate}
\end{thm}

\medskip
It was shown in \cite{FJLRYZ} for the NSE that unless a nonzero force is in a special class, $0$ cannot be on the global attractor.  
We next consider this matter for Galerkin approximations in the limit as $G\to \infty$, i.e., what are the consequences for steady states should $v=0$ with $g\neq 0$?  

For convenience, we use the short-hand writing
$$\xi_n \sim \eta_n,\text{ respectively, } \xi_n  \succ \eta_n,$$ for sequences $(\xi_n)_{n=1}^\infty$ and $(\eta_n)_{n=1}^\infty$, to mean 
$$(\xi_n)_{n=1}^\infty \sim (\eta_n)_{n=1}^\infty, \text{ respectively, }  
(\xi_n)_{n=1}^\infty \succ(\eta_n)_{n=1}^\infty.$$

In the case $\sigma_{1,1}\sim \sigma_0$, let 
\begin{equation}\label{xn}
    \lim_{n\to\infty} \alpha_n \Gamma_{1,n}^2= \lambda_*\in(0,\infty)\text{ and set }\chi_n=1-\alpha_n\Gamma_{1,n}^2/\lambda_*,
\end{equation}
and, by the virtue of Lemma \ref{concat} in the Appendix and using a subsequential set, we further assume that the sequence $(\chi_n)_{n=1}^\infty$ satisfies
either

(S1) $\chi_n=0$ for all $n$, or

(S2) $\chi_n>0$ for all $n$ and $\mathcal S\cup \{(\chi_n)_{n=1}^\infty\}$ is totally comparable, or

(S3) $\chi_n<0$ for all $n$ and $\mathcal S\cup \{(-\chi_n)_{n=1}^\infty\}$ is totally comparable.

In the case $v=0$ 
certain orderings of the sequences in \eqref{seqarray} lead to certain algebraic relations for the expansion vectors. Although the $H$-norm $|\cdot|$ and $V$-norm $\|\cdot\|$ are equivalent in $\mathcal H$, note that the 
norm $\|w_1\|$ that appears twice in the next result is specifically the $V$-norm.

\begin{prop} \label{v0eg} Assume $v=0$. 
Then $w_2$ exists in \eqref{infser}, $\sigma_{1,1} \succsim \sigma_0$ and only the following mutually exclusive cases can hold.
\begin{enumerate}[label=\tnum]
    \item\label{cas1} Case $\sigma_{1,1} \succ \sigma_0$. Then 
    $B(w_1,w_1)=0$ and $\sigma_{1,2} \succsim \sigma_0.$
    In addition,
    \begin{enumerate}[label=\nnum]
        \item\label{f1} if $\sigma_{1,2} \succ \sigma_0$, then $\Bs(w_1,w_2)=0$, while 
        \item\label{f2} if $\sigma_{1,2} \sim \sigma_0$, then $\lambda \Bs(w_1,w_2)=g$ for some $\lambda>0$ and $\langle g,w_1\rangle=0$.
    \end{enumerate}
    \smallskip
    \item\label{cas2} Case $\sigma_{1,1} \sim\sigma_0$. Then 
    $\lambda_* B(w_1,w_1)=g$   and $\sigma_{0,2} \succsim \sigma_{0}.$

    More specific cases are listed below.
    \smallskip
    \begin{enumerate}[label=\nnum]
        \item\label{f3} Case $\sigma_{0,2} \sim \sigma_0$. Then $\chi_n=0$ for all $n$, or $\Gamma_{1,n} \succsim |\chi_n|$. 
            \smallskip
        \begin{enumerate}[label=\rnum]
            \item\label{last3} If $\Gamma_{1,n} \sim |\chi_n|$, then 
            $$\lambda_1Aw_1 +\lambda_2\Bs(w_1,w_2)=g  \text{ for some } \lambda_1, \lambda_2\in\mathbb R \text{ with } \lambda_1\lambda_2>0,$$ and 
            $\lambda_* \lambda_1\|w_1\|^2 = \lambda_2\langle g,w_2\rangle$,
            
            \item\label{last4} If $\chi_n=0$ for all $n$,  or $\Gamma_{1,n} \succ |\chi_n|$, then $$Aw_1 +\lambda_2\Bs(w_1,w_2)=0\quad \text{for some} \quad \lambda_2 > 0$$ 
            and $\lambda_* \|w_1\|^2 = \lambda_2 \langle g,w_2\rangle$.
        \end{enumerate}
    \smallskip
        \item\label{f4} Case $\sigma_{0,2} \succ \sigma_0$.
        Then $\chi_n=0$ for all $n$, or $\alpha_n \Gamma_{1,n}\Gamma_{2,n} \succsim |\chi_n|$. 
        \begin{enumerate} [label=\rnum]   \smallskip        
            \item\label{last1} If $\alpha_n \Gamma_{1,n}\Gamma_{2,n} \sim |\chi_n|$, then 
            $\lambda_2\Bs(w_1,w_2)=g$ for some $\lambda_2 \neq 0,$
           \item\label{last2} If $\chi_n=0$ for all $n$, or $\alpha_n \Gamma_{1,n}\Gamma_{2,n} \succ |\chi_n|$, then $\Bs(w_1,w_2)=0.$
        \end{enumerate} 
    One has, in both cases \ref{last1} and \ref{last2}, that 
    $\langle g,w_2 \rangle =\langle B(w_2,w_2),w_1 \rangle  = 0$.
    \end{enumerate}
\end{enumerate}
\end{prop}

We saw in Theorem \ref{thm1} that we always have $B(v,v)=0$, and moreover, if the expansion is trivial, then $Av=g$.  We now consider the case where $Av=g$, but the expansion is nontrivial.  
The next result can also be seen as a further investigation of Theorem \ref{thm1}\ref{p2}(a). 

\begin{prop} \label{vAinvg} 
Assume $v=A^{-1}g$ and the expansion \eqref{infser} is nontrivial.
Then 
\begin{equation}\label{B1}
    \Bs(v,w_1)=0
\end{equation}
and there exists $w_2$ in \eqref{infser}. In addition, comparing $\sigma_1$, $\sigma_{0,2}$ and $\sigma_{1,1}$ gives only the following possibilities.

\begin{enumerate}[label=\tnum]
   \item\label{m1} The cases $\sigma_1 \succ \sigma_{0,2} ,\sigma_{1,1} $ and $\sigma_1 \sim \sigma_{1,1} \succ \sigma_{0,2}  $ are impossible.

    \item\label{m2} If $\sigma_{0,2} \succ \sigma_1 ,\sigma_{1,1} $, then
    $\Bs(v,w_2)=0.$

    \item\label{m3} If $\sigma_{1,1} \succ \sigma_{0,2} ,\sigma_1 $, then
    $B(w_1,w_1)=0.$

    \item\label{m4} If $\sigma_1 \sim  \sigma_{0,2} \succ \sigma_{1,1} $, then
    $Aw_1+\lambda \Bs(v,w_2)=0,$ for some $\lambda>0.$

    \item\label{m5} If $\sigma_{0,2} \sim  \sigma_{1,1} \succ \sigma_1  $, then
    $\Bs(v,w_2)+\lambda B(w_1,w_1)=0$, for some $\lambda> 0.$
    
    \item\label{m6} If $\sigma_1 \sim \sigma_{0,2} \sim  \sigma_{1,1}  $, then
    $Aw_1+\lambda_1\Bs(v,w_2)+\lambda_2 B(w_1,w_1)=0$, for some $\lambda_1,\lambda_2>0.$
\end{enumerate}
\end{prop}

\section{General expansion results}\label{rigor}

In this section, we introduce a new type of asymptotic expansion in general normed spaces.   Their existence is established in Lemma \ref{mainlem}. Numerical results for the 3D NSE are presented in subsection \ref{subnum} which demonstrate both the construction of the expansion as well as a particular case of Theorem \ref{thm1}.
Examples showing a variety of such expansions are given subsection \ref{sueeg}.
\begin{defn}\label{uexp}
Let $(Z,\|\cdot\|_Z)$ be a normed space over $\mathbb C$ or $\mathbb R$.
    We say a sequence $(v_n)_{n=N_0}^\infty$ in $Z$, for some integer $N_0\ge 1$,  has a strict unitary expansion if it satisfies one of the following three exclusive conditions.
\begin{enumerate}[label=\nnum]
    \item\label{d1} There is $v\in Z$ such that $v_n=v$ for all $n\ge N_0$.

    \item\label{d2} There are integer $K\ge 1$, vector $v\in Z$, unit vectors $w_k\in Z$, unit vectors $w_n^{(k)}\in Z$ and positive numbers $\Gamma_{k,n}$, for $n\ge N_0$ and $1\le k\le K$, such that 
\begin{align}
\label{b0}
\lim_{n\to\infty} \Gamma_{1,n}&=0,\\
\label{b1}
\lim_{n\to\infty} \frac{\Gamma_{k+1,n}}{\Gamma_{k,n}}&=0 \\
\intertext{ for all   $1\le k< K$,}
\label{b3}
\lim_{n\to\infty} w_n^{(k)}&=w_k 
\end{align}
for all $1\le k\le K$,
\begin{equation}\label{b4}
    v_n=v+\sum_{j=1}^{k-1} \Gamma_{j,n}w_j+\Gamma_{k,n}w_n^{(k)}
\end{equation}
for all $n\ge N_0$, $1\le k\le K$, and 
\begin{equation}\label{b6}
    w_n^{(K)}=w_{K} \text{ for all $n\ge N_0$.}
\end{equation}

   \item\label{d3}  
 There are a vector $v\in Z$, unit vectors $w_k\in Z$, vectors $w_n^{(k)}\in Z$,  numbers $\Gamma_{k,n}> 0$  and integers $N_k\ge N_0$ for all $n\ge N_0$, $k\ge 1$  such that one has \eqref{b0}, while \eqref{b1}, \eqref{b3}  hold for all $k\ge 1$, and  \eqref{b4} holds for all $k\ge 1$, $n\ge N_0$, and  each $w_n^{(k)}$ is a unit vector in $Z$ for $k\ge 1$ and $n\ge N_k$.
\end{enumerate}

We say the expansion in Case \ref{d1} is trivial, the expansions in Cases \ref{d1} and \ref{d2} are finite, and 
the expansion in Case \ref{d3} is infinite.
\end{defn}

We denote the strict unitary expansions by  
\begin{equation}\label{vsum}
    v_n\approx v+\sum_k \Gamma_{k,n} w_k,
\end{equation}
where the last summation can be void or finite or infinite.

Clearly, \eqref{vsum} implies
\begin{equation}\label{vlim}
    \lim_{n\to\infty} v_n=v,
\end{equation}
and $\Gamma_{k,n}$, when exists, satisfies
\begin{equation*}
\lim_{n\to\infty} \Gamma_{k,n}=0.
\end{equation*}

If $(v_n)_{n=N_0}^\infty$ has the strict unitary expansion \eqref{vsum}, then any subsequence $(v_{n_j})_{j=1}^\infty$
has the strict unitary expansion
\begin{equation}\label{vsubsum}
   v_{n_j}\approx v+\sum_k \Gamma_{k,n_j} w_k.
\end{equation}

Thanks to property \eqref{vsubsum}, whenever we use a subsequential set of $\mathcal S$ in Section \ref{results}, the strict unitary expansion \eqref{infser} is still valid.

The motivation for Definition \ref{uexp} is the following result.

\begin{lem}[Expansion lemma]\label{mainlem} 
    Any bounded sequence in a finite dimensional normed space over $\mathbb C$ or $\mathbb R$ has a subsequence that possesses a strict unitary expansion.
\end{lem}
\begin{proof}   
Let $(v_n)_{n=1}^\infty$ be a bounded sequence in a finite dimensional normed space $(Z,\|\cdot\|_Z)$ over $\mathbb C$ or $\mathbb R$.
By the relative compactness of bounded subsets of $Z$, $(v_n)_{n=1}^\infty$ 
has a convergent subsequence.
Hence there exist $v\in Z$ and a subsequence $(\varphi_0(n))_{n=1}^\infty$ of $(n)_{n=1}^\infty$ such that
\begin{equation}\label{vnv}
    \lim_{n\in\varphi_0(\mathbb N), n\to \infty} v_n=v.
\end{equation} 

If $v_n=v$ for all sufficiently large $n$, then we have Case \ref{d1} in Definition \ref{uexp} for a subsequence of $v_n$. 

Otherwise, define $\gamma_n^{(1)}=\|v_n-v\|_Z$ for $n\in\varphi_0(\mathbb N)$. Then there is a subsequence $(\varphi_1(n))_{n=1}^\infty$ of $\varphi_0(n)_{n=1}^\infty$ such that
$\gamma_n^{(1)}>0$ for all $n\in\varphi_1(\mathbb N)$. Define 
\begin{equation*}
w_n^{(1)}=\frac{1}{\gamma_n^{(1)}}(v_n-v)\text{ for all }n\in \varphi_1(\mathbb N).    
\end{equation*}

Then, for all $n\in \varphi_1(\mathbb N)$,
\begin{equation}\label{w2}
    v_n=v+\gamma_n^{(1)}w_n^{(1)}, \
    \|w_n^{(1)}\|_Z=1,\ 
    \lim_{n\in\varphi_1(\mathbb N),n\to\infty}\gamma_n^{(1)}=0.
\end{equation}

By compactness of the unit sphere in $Z$, there exists a 
subsequence of $(\varphi_1(n))_{n=1}^\infty$, still denoted by $(\varphi_1(n))_{n=1}^\infty$  such that 
\begin{equation*}
    \lim_{n\in\varphi_1(\mathbb N),n\to\infty}w_n^{(1)}=w_1\in Z,
\end{equation*}
where the limit defines $w_1$. Clearly, $\|w_1\|_Z=1$. 

Denote $w_0=v$ and let $w_n^{(0)}=v_n$ for all $n\ge 1$.

For $k\ge 1$, we define the following statement $(T_k)$.

($T_k$) \emph{There are subsequences $(\varphi_j(n))_{n=1}^\infty$ of $(n)_{n=1}^\infty$, for $1\le j\le k$, with each $(\varphi_{j}(n))_{n=1}^\infty$ being a subsequence of $(\varphi_{j-1}(n))_{n=1}^\infty$, unit vectors $w_j$, positive numbers $\gamma_n^{(j)}$ and unit vectors $w_n^{(j)}$, for $n\in\varphi_j(\mathbb N)$, defined by
\begin{equation*}
    \gamma_n^{(j)}=\|w_n^{(j-1)}-w_{j-1}\|_Z,
\end{equation*}
and
\begin{equation*}
w_n^{(j)}=\frac1{\gamma_n^{(j)}}(w_n^{(j-1)}-w_{j-1}), \text{ i.e., }
w_n^{(j-1)}=w_{j-1}+\gamma_n^{(j)}w_n^{(j)},
\end{equation*}
such that, for $1\le j\le k$, 
\begin{equation}\label{w6}
    \lim_{n\in\varphi_j(\mathbb N),n\to\infty}w_n^{(j)}=w_j\in Z,
\end{equation}
and
\begin{equation}\label{w7}
v_n=v+\gamma_n^{(1)}w_1+\gamma_n^{(1)}\gamma_n^{(2)}w_2+\cdots+
\gamma_n^{(1)}\cdots\gamma_n^{(j-1)}w_{j-1}+\gamma_n^{(1)}\cdots\gamma_n^{(j)}w_n^{(j)}
\end{equation}
for $n\in\varphi_j(\mathbb N)$.
} 

Observe, when ($T_k$) holds true and $1\le j\le k$, that 
\begin{equation}\label{gamlim}
    \lim_{n\in\varphi_j(\mathbb N),n\to\infty}\gamma_n^{(j)}=0.
\end{equation}
Indeed, \eqref{gamlim} is true for $j=1$ thanks to the limit in \eqref{w2}, and for $1<j\le k$, we have from \eqref{w6} that
$$ \lim_{n\in\varphi_{j}(\mathbb N),n\to\infty}\gamma_n^{(j)}=\lim_{n\in\varphi_{j-1}(\mathbb N),n\to\infty}\gamma_n^{(j)}
=\lim_{n\in\varphi_{j-1}(\mathbb N),n\to\infty}\|w_n^{(j-1)}-w_{j-1}\|_Z=0.$$

We already verified that ($T_1$) holds.
Let $k\ge 1$. Assume ($T_k$) holds.
In particular, with $j=k$ in \eqref{w7}, we have, for $n\in \varphi_k(\mathbb N)$,
\begin{equation}\label{w8}
  v_n=v+\gamma_n^{(1)}w_1+\gamma_n^{(1)}\gamma_n^{(2)}w_2+\cdots+
\gamma_n^{(1)}\cdots\gamma_n^{(k-1)}w_{k-1}+\gamma_n^{(1)}\cdots\gamma_n^{(k)}w_n^{(k)}.
\end{equation}

\medskip
\emph{Case 1. $w_n^{(k)}=w_k$ for all large $n\in \varphi_k(\mathbb N)$.} 
Set $K=k$.
Then there is a subsequence $(\varphi(n))_{n=1}^\infty$ of $(\varphi_k(n))_{n=1}^\infty$ such that
\begin{equation*}
    w_n^{(K)}= w_{K}\text{ for all }n\in\varphi(\mathbb N).
\end{equation*}

Let $V_n=v_{\varphi(n)}$, and for $1\le j\le K$, $\Gamma_{j,n}=\gamma_{\varphi(n)}^{(1)}\gamma_{\varphi(n)}^{(2)}\ldots \gamma_{\varphi(n)}^{(j)}$, and $\widetilde w_n^{(j)}=w_{\varphi(n)}^{(j)}$. 

For $0\le j\le K$, $(\varphi(n))_{n=1}^\infty$ is a subsequence of $(\varphi_j(n))_{n=j}^\infty$.
Hence $(V_n)_{n=1}^\infty$ is a subsequence of $(v_n)_{n=1}^\infty$, and more generally, of  $( v_{\varphi_j(n)} )_{n=1}^\infty$ for $0\le j\le K$. We have from the limit in \eqref{w2} that
\begin{equation}\label{Gam1}
    \Gamma_{1,n}=\gamma_{\varphi(n)}^{(1)}\to 0\text{ as }n\to\infty.
\end{equation}
As a consequence of \eqref{gamlim}, we also have, for all $1\le j\le K$,
\begin{equation}\label{gamlim2}
    \lim_{n\to\infty}\gamma_{\varphi(n)}^{(j)}=0.
\end{equation}
If $K>1$ and $1\le j<K$, it follows \eqref{gamlim2} that 
\begin{equation}\label{limGrat}
\lim_{n\to\infty}\frac{\Gamma_{j+1,n}}{\Gamma_{j,n}}=\lim_{n\to\infty}\gamma_{\varphi(n)}^{(j+1)}=0.
\end{equation}
By \eqref{w6}, it holds for $1\le j\le K$ that
\begin{equation}\label{w16}
    \lim_{n\to\infty}\widetilde w_n^{(j)}=w_j.
\end{equation}
By \eqref{w7}, we have, for $1\le j\le K$ and $n\ge 1$,
\begin{equation}\label{w17}
V_n=v_{\varphi(n)}=v+\Gamma_{1,n}w_1+\cdots+
\Gamma_{j-1,n}w_{j-1}+\Gamma_{j,n}\widetilde w_n^{(j)},
\end{equation}
and $\widetilde w_n^{(K)}=w_K$ for $n\ge 1$.
Thus we obtain Case \ref{d2} in Definition \ref{uexp} for the sequence $(V_n)_{n=1}^\infty$ 
with $\widetilde w_n^{(j)}$ replacing $w_n^{(j)}$; we stop the induction.

\medskip
\emph{Case 2. There is a subsequence $(\varphi_{k+1}(n))_{n=1}^\infty$ of $(\varphi_k(n))_{n=1}^\infty$ such that
$w_n^{(k)}\ne w_k$ for all $n\in\varphi_{k+1}(\mathbb N)$.}
Define, for all $n\in\varphi_{k+1}(\mathbb N)$,  $\gamma_n^{(k+1)}=\|w_n^{(k)}-w_k\|_Z>0$ and 
\begin{equation}\label{w9}
w_n^{(k+1)}=\frac1{\gamma_n^{(k+1)}}(w_n^{(k)}-w_k), \text{ which yields }
w_n^{(k)}=w_k+\gamma_n^{(k+1)}w_n^{(k+1)}.    
\end{equation}
Substituting  the last equation of \eqref{w9} into \eqref{w8}, one has 
\begin{equation*}
  v_n=v+\gamma_n^{(1)}w_1+\gamma_n^{(1)}\gamma_n^{(2)}w_2+\cdots+
\gamma_n^{(1)}\cdots\gamma_n^{(k)}w_{k}+\gamma_n^{(1)}\cdots\gamma_n^{(k+1)}w_n^{(k+1)}.
\end{equation*}
Also, $\|w_n^{(k+1)}\|_Z=1$ for $n\in \varphi_{k+1}(\mathbb N)$.

By the compactness of the unit sphere in $Z$ again, we can extract a subsequence of $(\varphi_{k+1}(n))_{n=1}^\infty$, but still denoted by $(\varphi_{k+1}(n))_{n=1}^\infty$,  such that
\begin{equation*}
    \lim_{n\in \varphi_{k+1}(\mathbb N),n\to\infty}w_n^{(k+1)}=w_{k+1}\text{ with } \|w_{k+1}\|_Z=1.
\end{equation*}
Thus, ($T_{k+1}$) holds.

If the induction does not stop at any step $k$, then by the Induction Principle, we have ($T_k$) holds for all $k\ge 1$. 
Let $\varphi(n)=\varphi_n(n)$ and  $V_n=v_{\varphi(n)}$.
For $j\ge 1$, $(\varphi(n))_{n=j}^\infty$ is a subsequence of $(\varphi_j(n))_{n=j}^\infty$.
Hence $(V_n)_{n=1}^\infty$ is a subsequence of $(v_n)_{n=1}^\infty$, and more generally, $(V_n)_{n=j}^\infty$ is a subsequence of of  $( v_{\varphi_j(n)} )_{n=j}^\infty$ for $j\ge 1$. 
We define positive numbers $\Gamma_{k,n}$ and vectors $\widetilde w_n^{(k)}\in Z$ as follows.

For $k=1$, define 
$$\Gamma_{1,n}=\gamma_{\varphi(n)}^{(1)}\text{ and } \widetilde w_n^{(1)}=w_{\varphi(n)}^{(1)}\text{  for all }n\ge 1.$$

For $k\ge 2$, when $1\le n <k$ let
$$z_{k,n}=V_n-v-\sum_{j=1}^{k-1}\Gamma_{j,n}w_j$$
and define 
\begin{equation}\label{GWsmall}
\left\{     \begin{aligned}
&\Gamma_{k,n}=\|z_{k,n}\|_Z, \quad \widetilde w_n^{(1)}=\frac{1}{\Gamma_{k,n}}z_{k,n} &&\text{ in the case }z_{k,n} \ne 0,\\
&\Gamma_{k,n}=\frac1{2^{nk}} \Gamma_{k-1,n}, \quad \widetilde w_n^{(1)}=0 &&\text{ in the case }z_{k,n} = 0;
\end{aligned}
\right.
\end{equation}
while when $n\ge k$, define
\begin{equation}\label{GWlarge}
\Gamma_{k,n}=\gamma_{\varphi(n)}^{(1)}\gamma_{\varphi(n)}^{(2)}\ldots \gamma_{\varphi(n)}^{(k)} \text{ and }
\widetilde w_n^{(k)}=w_{\varphi(n)}^{(k)}.
\end{equation}

Note that we still have the limit \eqref{Gam1}, while the limit \eqref{gamlim2} is true for all $j\ge 1$, which implies that property \eqref{limGrat} now holds true for all $j\ge 1$. Moreover, the limit \eqref{w16} is true for all $j\ge 1$.

Set $N_k=k$ for all $k\ge 1$. For $k\ge 1$, by \eqref{GWlarge} and ($\mathcal T_k$), we have $\|\widetilde w_{n}^{(k)}\|_Z=1$ for all $n\ge N_k$.
By \eqref{w7}, we obtain \eqref{w17} for $j\ge 1$, $n\ge N_j$.
By definition \eqref{GWsmall}, we also obtain \eqref{w17} for $n< N_j$. Thus, the identity \eqref{w17} holds true for all $j\ge 1$ and $n\ge 1$.
Therefore, we obtain Case \ref{d3} in Definition \ref{uexp} for the sequence $(V_n)_{n=1}^\infty$ 
with $\widetilde w_n^{(k)}$ replacing $w_n^{(k)}$.
The proof of Lemma \ref{mainlem} is complete.
\end{proof}

\subsection{A numerical demonstration for the 3D NSE}\label{subnum}
We next demonstrate the algorithm for the expansion in the proof of Lemma \ref{mainlem} numerically for the 3D NSE, with periodic boundary conditions.  Considering equation \eqref{steady}, to generate a sequence of steady states, we choose one, 
$$
\bu(x,y,z)=\begin{bmatrix}\sin(z)\\\sin(x)\\0\end{bmatrix} ,
$$
to start, and let that define the fixed balancing force $g=\bg(\cdot)$ at $\alpha_1=1$
$$
\bg(x,y,z)=-\bu+\begin{bmatrix} 0 \\ \sin(z)\cos(x) \\ 0 \end{bmatrix}=-\bu+
\frac{1}{2}\begin{bmatrix} 0 \\  \sin(z+x)+\sin(z-x) \\ 0 \end{bmatrix},
$$
i.e., the Fourier coefficients are $\hat \bg_\bk= -\hat \bu_\bk +\hat \bB_\bk $ where 
$\hat\bu_\bk=\hat\bB_\bk=\b0$ except for
$$
\hat{\bu}_{\mathbf e_1} =-\frac{i}{2}\mathbf e_2,\quad 
\hat{\bu}_{\mathbf e_3} =-\frac{i}{2}\mathbf e_1,  \quad \text{and} \quad
\hat{\bB}_{(1,0,1)}=\hat{\bB}_{(-1,0,1)}=-\frac{i}{4} \mathbf e_2,
$$
and the reality condition $\hat\bu_{-\bk}=\overline{\hat\bu}_\bk$,
$\hat\bB_{-\bk}=\overline{\hat\bB}_\bk$.

Note, in this case, that the $H$-norm $|\bg|=\pi\sqrt{10\pi}\ne 1$. Then, strictly speaking, the Grashof number is $G_n=\alpha_n|\bg|$.

We use a small Galerkin approximation ($\Pi_N=\bar P_9$) so we can work with a fairly long sequence. Since the number of modes is small, we compute the Fourier coefficients of the bilinear term $B(\bu,\bu)$ directly through the convolution
$$
\hat \bB_\bk =i\sum_{|\bj|\le 3}  (\bk\cdot \hat \bu_\bj) \hat \bu_{\bk-\bj} 
-(\bk\cdot \hat \bu_\bj)(\bk\cdot \hat \bu_{\bk-\bj})\frac{\bk}{|\bk|^2}
$$
as opposed doing so by fast Fourier transform.
By the reality condition, we may work with half the modes, 
which in this case results in a dynamical system with 366 real variables. 
Using the bifurcation software package AUTO \cite{Doedel}, we compute a branch of steady states, whose norms are plotted in Figure \ref{Bif6}. Due to the length factor in the Parseval identity, and using only half the modes, we relate
$$
\|v\|_Z^2=\frac12 \sum_{|\bk| \le 3} |\hat {\mathbf v}_{\bk}|^2= \frac{1}{16\pi^3} |v|^2,
$$ 
when considering $Z=\mathcal H$ in Lemma \ref{mainlem}.

The computed  sequence of steady states $v_n$ appears to converge,  as $\alpha_n$ increases, to a limit close to $v_{n^*}$, where $\alpha_{n^*}=200000$.  We then set $\gamma_n^{(1)}=\|v_n-v_{n^*}\|_Z$ for $n \le n^*-1$ and 
\begin{equation*}
w_n^{(1)}=\frac{1}{\gamma_n^{(1)}}(v_n-v_{n^*})\text{ for all }n\le n^*-1.  \end{equation*}
The sequence $w_n^{(1)}$ also appears to converge for this data; no subsequence is needed.  We take as its limit $w_{n^*-1}^{(1)}$ and set $\gamma_n^{(2)}=\|w_n^{(1)}-w_{n^*-1}^{(1)}\|_Z$ for $n \le n^*-2$. 
This enables us to plot in Figure \ref{sigfig6} the first six nontrivial elements in the left corner of \eqref{seqarray}.  The rapid decay for each sequence that starts around $G= 190000$ is likely due to taking $v_{n^*}$ in place of the true limit $v$ (and similarly for $w_{n^*-1}^{(1)}$).  AUTO uses Newton iteration to continue the sequence of steady states.  We set its tolerance to both $10^{-7}$ and $10^{-9}$ and found no perceptible difference in the resulting plots in Figure \ref{sigfig6}. 
We plot in Figure \ref{ratfig6} the ratios of pairs of these sequences which suggest the following ordering:
$$
\sigma_{0,0} \succ \sigma_0 \sim \sigma_{0,1}\succ\sigma_1 \sim\sigma_{1,1}\sim \sigma_{0,2}\succ \sigma_2 \sim \sigma_{1,2}.$$
The ratio $\sigma_{0,1}/\sigma_0=\sigma_{0,1}$ which is shown in Figure \ref{sigfig6}, 
is nearly constant, until $\alpha_n > 180000$.    If instead,  we were to interpret the plot of 
$\sigma_{0,1}/\sigma_0=\sigma_{0,1}$ to be truly starting to decay toward 0, then we would have case (iii), (a) of Theorem \ref{thm1}. But we find that $\|Av_{n^*}-g\|_Z\approx 0.66$ which would contradict \eqref{Avg}.  We conclude that $\sigma_0 \sim \sigma_{0,1}$ and hence that this example demonstrates case (iii), (c) of Theorem \ref{thm1}. The next relation, $\sigma_{0,1}\succ\sigma_1$, must hold.  This is confirmed by the clear decay in the plot of $\sigma_{1}/\sigma_{0,1}$.  Since $\sigma_0 \sim \sigma_{0,1}$, we must also have $\sigma_1 \sim\sigma_{1,1}$ and $\sigma_2 \sim \sigma_{1,2}$ (the latter confirmed by the plot of $\sigma_{1,2}/\sigma_{2}$).  Meanwhile, the plot of $\sigma_{1}/\sigma_{0,2}$, being fairly flat, suggests $\sigma_{1}\sim\sigma_{0,2}$.

\begin{figure}[ht]
\psfrag{100000}{$10^6$}
\psfrag{G}{$\alpha_n$}
\psfrag{L2}{$\|v_n\|_Z$}
\centerline{\includegraphics[scale=1.3]{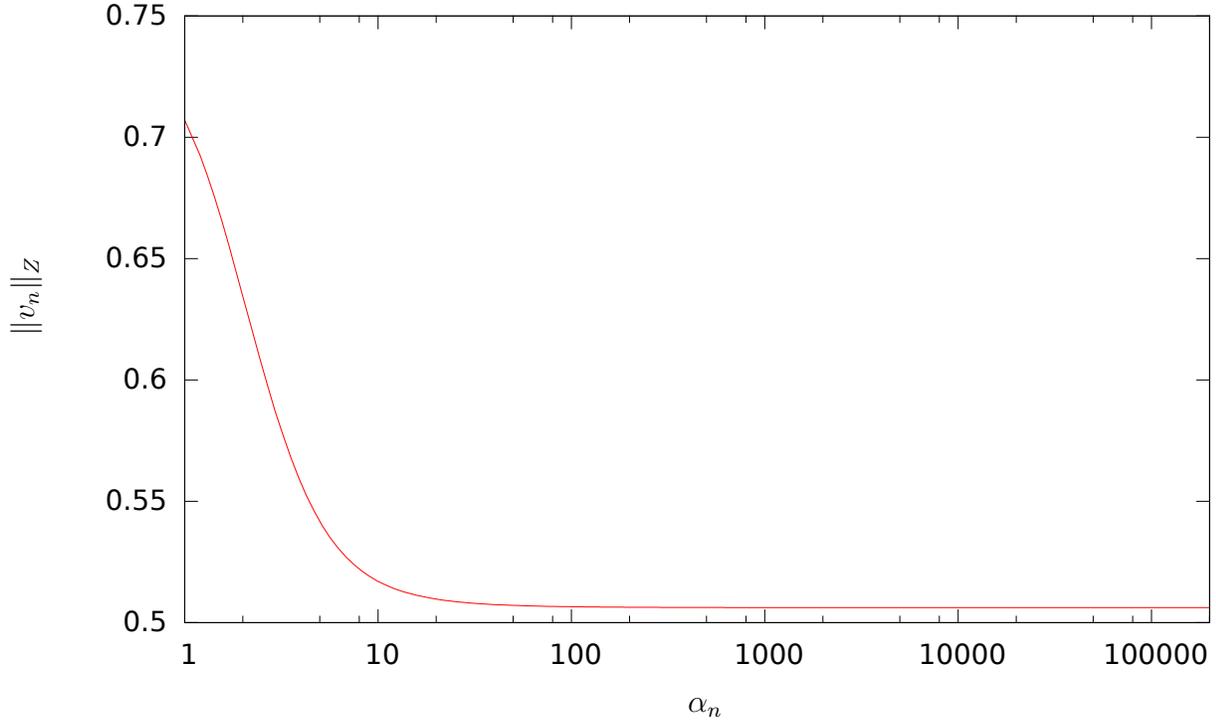}  }
\vskip .15 truein
\caption{Bifurcation diagram.}
\label{Bif6}
\end{figure}

\begin{figure}[ht]
\psfrag{'fort.12'}{\tiny $\sigma_1$}
\psfrag{'fort.13'}{\tiny$\sigma_{0,1}$}
\psfrag{'fort.14'}{\tiny$\sigma_{1,1}$}
\psfrag{'fort.15'}{\tiny$\sigma_{2}$}
\psfrag{'fort.16'}{\tiny$\sigma_{0,2}$}
\psfrag{'fort.17'}{\tiny$\sigma_{1,2}$}
\psfrag{G}{$\alpha_n$ }
\centerline{
\includegraphics[scale=1.2]{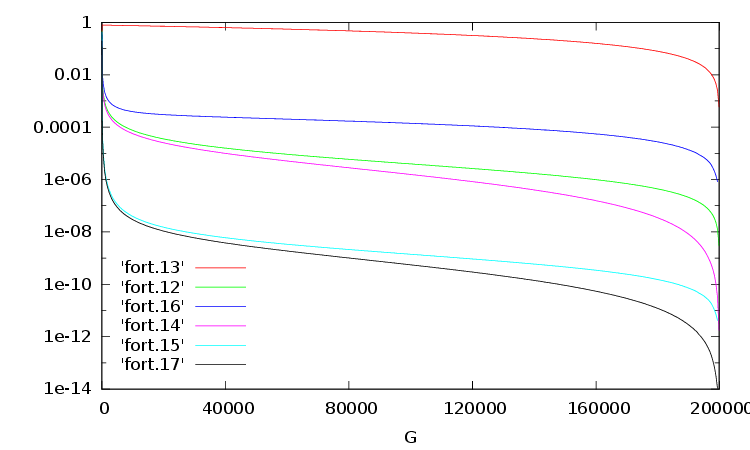}   }
\vskip .15 truein
\caption{Sequences in \eqref{seqarray}.}
\label{sigfig6}
\end{figure}

\begin{figure}[ht]
\psfrag{'fort.18'}{\tiny $\sigma_1/\sigma_{0,1}$}
\psfrag{'fort.19'}{\tiny$\sigma_{1}/\sigma_{0,2}$}
\psfrag{'fort.20'}{\tiny$\sigma_{1,1}/\sigma_{0,2}$}
\psfrag{'fort.21'}{\tiny$\sigma_2/\sigma_{1,1}$}
\psfrag{'fort.22'}{\tiny$\sigma_{1,2}/\sigma_{2}$}
\psfrag{G}{$\alpha_n$}
\centerline{
\includegraphics[scale=1.2]{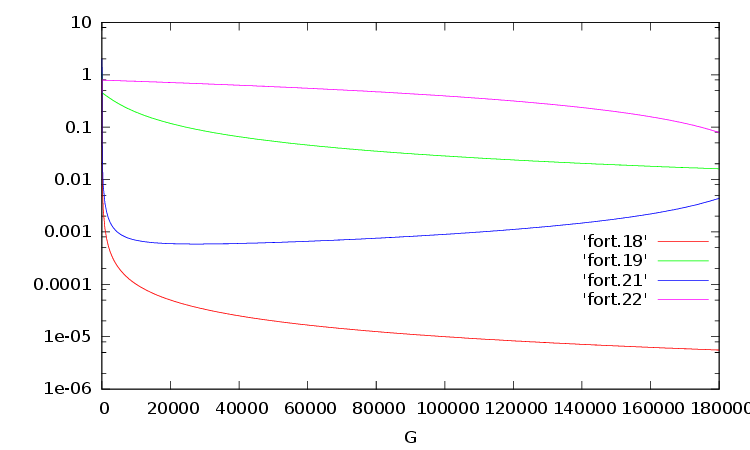}   }
\vskip .15 truein
\caption{Sequence ratios.}
\label{ratfig6}
\end{figure}

\subsection{On the uniqueness of strict unitary expansions}

Because there are no convergence rates specified for $\Gamma_{k,n}$, as $n\to\infty$, a sequence $(v_n)_{n=1}^\infty$ may be written as different sums in the form of the right-hand side of \eqref{vsum}. However, thanks to the condition $\|w_n^{(k)}\|_Z=1$, for sufficiently large $n$, it turns out that the strict unitary expansions are unique for sufficiently large $n$. More precisely, we have the following ``asymptotic uniqueness".

\begin{prop}\label{unique}
Let $(v_n)_{n=N_0}^\infty$ have a strict unitary expansion as in Definition \ref{uexp}.
 Denote
 $$\bar N_k=\begin{cases}
     N_0, & \text{ in Case \ref{d2} for } 1\le k\le K,\\
     \max\{N_1,N_2,\ldots,N_k\}, & \text{ in Case \ref{d3} for all } k\ge 1.
 \end{cases}
 $$
 Then the vectors $v$, $w_k$, $w_n^{(k)}$ and positive numbers $\Gamma_{k,n}$  are uniquely determined for $k\ge 1$ and $n\ge \bar N_k$.
\end{prop}
\begin{proof}
Suppose \eqref{vsum} is a strict unitary expansion in a normed space $(Z,\|\cdot\|_Z)$ over $\mathbb C$ or $\mathbb R$.
Then $v$ is determined by the limit \eqref{vlim}.
If $v_n=v$ for all $n$, then this is the unique expansion for $v_n$.
Otherwise, for  {$n\ge \bar N_1$}, we have from \eqref{b4} that 
$$\|v_n-v\|_Z=\Gamma_{1,n}\|w_n^{(1)}\|_Z=\Gamma_{1,n},$$
hence $\Gamma_{1,n}$ is uniquely determined, and then
$w_n^{(1)}=(v_n-v)/\Gamma_{1,n}$ is uniquely determined. Subsequently, $w_1=\lim_{n\to\infty} w_n^{(1)}$ is uniquely determined.

Recursively, for $k\ge 1$, suppose $w_j$,  $w_n^{(j)}$ and $\Gamma_{j,n}$ are already uniquely determined for $1\le j\le k$ and  {$n\ge \bar N_j$}.  
If $w_n^{(k)}=w_k$ eventually as $n\to\infty$, then the expansion must fall to Case \ref{d2} in Definition \ref{uexp} and the number $K$ is $k$ which means $K$ is uniquely determined. 
In this case, $\bar N_j=N_0$ for all $1\le j\le K$.
Thus $v_n\approx v+\sum_{j=1}^{K}{\Gamma_{j,n}w_j}$ is the unique expansion of $v_n$.

Consider the case ``$w_n^{(k)}=w_k$ eventually as $n\to\infty$" does not hold. Then $w_{k+1}$, $\Gamma_{k+1,n}$, $w_n^{(k+1)}$ and  $\bar N_{k+1}$ exist. Let  {$n\ge \bar N_{k+1}$}. We have  {$n\ge \bar N_j$} for $1\le j\le k$, hence $\Gamma_{j,n}$ was already uniquely determined, and 
$$\left\|v_n-\sum_{j=1}^{k} \Gamma_{j,n}w_j\right \|_Z=\Gamma_{k+1,n}\|w_n^{(k+1)}\|_Z=\Gamma_{k+1,n}$$
which implies $\Gamma_{k+1,n}$ is uniquely determined. Thus, $w_n^{(k+1)}$ is uniquely determined for  {$n\ge \bar N_{k+1}$} by 
$$ w_n^{(k+1)} = \frac1{\Gamma_{k+1,n}}\left(v_n-\sum_{j=1}^{k} \Gamma_{j,n}w_j\right),$$
and then $w_{k+1}=\lim_{n\to\infty} w_n^{(k+1)}$  is uniquely determined.
\end{proof}

\subsection{Examples}\label{sueeg}
We present some examples of strict unitary expansions.
\begin{enumerate}[label=\rnum]
     \item\label{ea} Let $v_n\to v$ in $\mathbb R$. If $v_n>v$ for all $n\ge 1$ then $(v_n)_{n=1}^\infty$ has a unique strict unitary expansion 
     \begin{equation}\label{simex}
         v_n=v+\Gamma_{1,n}w_1,
     \end{equation}
     with $\Gamma_{1,n}=v_n-v$ and $w_1=1$.

     If $v_n<v$ for all $n\ge 1$ then $(v_n)_{n=1}^\infty$ has the unique strict unitary expansion \eqref{simex} with 
     $\Gamma_{1,n}=v-v_n$ and $w_1=-1$.

We can see that strict unitary expansions are not very useful in this case. 

\item\label{eb} Let $v_n=(-1)^n/n$ in $\mathbb R$. Suppose $(v_n)_{n=1}^\infty$ has a strict unitary expansion. Then the terms $\Gamma_{1,n}$, $w_1$ and $w_n^{(1)}$ exist in Definition \ref{uexp}. Clearly, $\Gamma_{1,n}=|v_n|=1/n$ and $w_n^{(1)}=v_n/\Gamma_{1,n}=(-1)^n$, which is divergent and contradicts Definition \ref{uexp}. Thus, $(v_n)_{n=1}^\infty$ does not have a strict unitary expansion.
However, its subsequences $v_{2n}=1/(2n)$ and $v_{2n+1}=-1/(2n+1)$  have finite strict unitary expansions as in part \ref{ea}.

 \item\label{ed} We give an example of infinite strict unitary expansions.
 Let $Z=\mathbb C$ and $v_n=e^{i/n}$. Clearly, we have the Taylor expansion
$v_n=1+\sum_{k=1}^\infty \frac{1}{k! n^k}i^k$.
However, the strict unitary expansion is a different one which is found below.
 Note, for $x\in \mathbb R$, that
$$e^{ix}=1+(e^{ix}-1)=1+e^{\frac{i}{2}x}(e^{\frac{i}{2}x}-e^{\frac{-i}{2}x})=1+2i\sin(x/2)e^{\frac{i}{2}x}.$$
Applying this formula repeatedly, we have
 \begin{align*}
     v_n&= 1+2i\sin\left (\frac{1}{2n}\right) e^{\frac{i}{2n}}  
     = 1+2i\sin\left (\frac{1}{2n}\right) +(2i)^2\sin\left (\frac{1}{2n}\right)\sin\left (\frac{1}{2^2n}\right) e^{\frac{i}{2^2n}} 
     =\ldots
 \end{align*}
and, for any $k\ge 1$,
\begin{equation}\label{vneg}
    v_n=1+\sum_{j=1}^k (2i)^j \prod_{p=1}^j \sin\left(\frac{1}{2^pn}\right)
+(2i)^{k+1}\prod_{p=1}^{k+1} \sin\left(\frac{1}{2^pn}\right) e^{\frac{i}{2^{k+1}n}}.
\end{equation}

Then we have the infinite strict unitary expansion \eqref{vsum} with
$$ w_k=i^k\text{ and }\Gamma_{k,n}=2^k \prod_{p=1}^k \sin\left(\frac{1}{2^pn}\right).$$ 
Indeed, letting $w_n^{(k)}=i^k e^{\frac{i}{2^{k}n}}$, one has $|w_k|=|w_n^{(k)}|=1$, $w_n^{(k)}\to w_k$,
$$\Gamma_{1,n}=2\sin\left(\frac{1}{2n}\right)\to 0\text{  and }
\frac{\Gamma_{k+1,n}}{\Gamma_{k,n}}=2 \sin\left(\frac{1}{2^{k+1}n}\right)\to 0\text{ as } n\to\infty.$$
Also, \eqref{vneg} implies \eqref{b4}. 
\end{enumerate}

\section{Proofs of main results} \label{mainproofs}

We present the proofs of Propositions \ref{Sprop}, \ref{v0eg}, \ref{vAinvg} and Theorem \ref{thm1}  in this section.
We recall that the assumptions throughout this section are $g\in \mathcal H\setminus\{0\}$, $\alpha_n>0$ for all $n$, $\lim_{n\to\infty}\alpha_n=\infty$, $v_n\in\mathcal H$ is a solution of \eqref{steady}, and $(v_n)_{n=1}^\infty$ has the strict unitary expansion \eqref{vsum} in $\mathcal H$.

\subsection{Proof of Proposition \ref{Sprop}}\label{pfprop}
\begin{proof}
In the case $(v_n)_{n=1}^\infty$ has a finite expansion, there are only finitely many sequences in \eqref{seqarray}. Hence,  both statements (i) and (ii) are obviously true. Consider below the case of an infinite expansion.

(i) Let $\mathcal{C}\in \widehat{\mathcal S}$.  If $\mathcal{C}$ contains more
than one sequence, it must contain one of the type $\sigma_{j,k}$.
Note that there can be at most one sequence equivalent to
$\sigma_{j,k}$ in each row of the array in \eqref{seqarray} up to
row $k+2$, i.e., the rows starting with $\sigma_0,\sigma_{0,0},\sigma_{1,1},\ldots,\sigma_{k,k}$. 
Also, $\sigma_{j,k}\succ\sigma$ for any $\sigma$ in the remaining rows.
Therefore, $\mathcal C$ has only finitely many elements.

(ii) Let $\mathcal{F}$ be a nonempty subset of $\widehat{\mathcal S}$. 
Consider the case when all $x$'s in $\mathcal F$ are of the type
$\langle\sigma_k\rangle$ for some $k\ge 0$. Let $k_0$ be the minimum of all such $k$'s.
Then $\langle\sigma_{k_0}\rangle$ is the minimum of $\mathcal F$.
Now, consider the case when there exists a sequence $\sigma_{j,k}$ such that
$\mathcal C=\langle\sigma_{j,k}\rangle\in \mathcal F$.
Define $\mathcal F_{\mathcal C}$ to be the set of $x\in\mathcal F$ such that $x\le \mathcal C$.
By \eqref{seqarray}, the intersection of $\mathcal F_{\mathcal C}$ with each $i$th-row, for $1\le i\le k+2$, starting with $\sigma_0,\sigma_{0,0},\sigma_{1,1},\ldots,\sigma_{k,k}$ is either empty or has a minimum $x_i$.
Also, the intersection of $\mathcal F_{\mathcal C}$ with  the remaining rows in \eqref{seqarray} is empty.
Since $\mathcal C\in \mathcal F_{\mathcal C}$, the set $\mathcal F_{\mathcal C}$ is not empty. 
Then the set of the above $x_i$'s is not empty. Comparing finitely many $x_i$'s, we find their minimum, which is the minimum of $\mathcal F_{\mathcal C}$, and hence of $\mathcal F$.
\end{proof}

\subsection{Proof of Theorem \ref{thm1}}\label{pf31}
\begin{proof} 

The facts about $\ord(\sigma_{0,0})$ and $\ord(\sigma)$ for  $\sigma\in \mathcal S\setminus\{\sigma_{0,0}\}$ in \ref{p0} are clear from the relations in \eqref{seqarray}.
Dividing \eqref{steady} by $\alpha_n$ and passing $n\to\infty$ yield \eqref{Bzero}.

 With $v_n=v$ and \eqref{Bzero}, we immediately have \eqref{Avg} from \eqref{steady}, giving \ref{p1}.

For \ref{p2},  we substitute the expansion $v_n=v+\Gamma_{1,n}w_n^{(1)}$ into the steady state equation \eqref{steady} and use property \eqref{Bzero} to obtain
    \begin{equation}\label{se1}
 Av-g + \alpha_n \Gamma_{1,n} \Bs(v,w_n^{(1)}) 
    + \alpha_n\Gamma_{1,n}\Gamma_{1,n}B(w_n^{(1)}, w_n^{(1)})=0 .       
    \end{equation}
For (a),(c) we can simply take $n\to \infty$ in \eqref{se1}, while for (b) we divide \eqref{se1} by $\alpha_n\Gamma_{1,n}$ and then take $n \to \infty$.  In each case the order of the sequences is clear.
    
    \medskip
For \ref{p3}, we note from the second row in \eqref{seqarray} that there are at least $k$ distinct numbers $\ord(\sigma_{0,j})$, for $0\le j\le k-1$, that are smaller than  $\ord(\sigma_{0,k})$.
Thus we obtain the first inequality in \eqref{zerok}. 

We prove the upper bound in \eqref{zerok} by contradiction. Assume the contrary, i.e.,
$$\ord(\sigma_{0,k}) > k(k+3)/2+1.$$
The possible entries $\sigma$ such that $\sigma \succ \sigma_{0,k}$ are $\sigma_0,\sigma_1,\ldots,\sigma_{k-1}$ along with 
$\sigma_{j,m}$ with $0\le j\le m\le k-1$, i.e., the triangle of sequences in \eqref{seqarray} with vertices $\sigma_{0,0}$,
$\sigma_{0,k-1}$, and $\sigma_{k-1,k-1}$.
Combining this fact with \eqref{orule}, we must have
\begin{equation}\label{in1}    
\ord(\{\sigma_0,\ldots,\sigma_{k-1},\sigma_{j,m}: 0\le j\le m\le k-1 \})\supset \{ 1,2,\ldots, k(k+3)/2+1\}.
\end{equation}

Because the first set has at most 
$$k  + k(k + 1)/2 = k(k + 3)/2$$ 
elements, the inclusion \eqref{in1} is impossible.

\medskip
The lower bound in \eqref{ojk} is obtained by using the chain
$$\sigma_{0,0}\succ \sigma_{0,1}\succ \ldots\succ \sigma_{0,k}
\succ \sigma_{1,k}\succ \ldots \succ \sigma_{j-1,k}\succ \sigma_{j,k}.$$

The proof of the upper bound in \eqref{ojk} is technical and will be presented in Appendix \ref{ap}.

\medskip
It remains to prove \ref{p5}. Using the chain $\sigma_{0,0}\succ\sigma_0\succ\ldots\succ\sigma_{k-1}\succ\sigma_k$, we immediately establish the lower bound of $\ord(\sigma_k)$ in \eqref{ok}. 
For the upper bound, denote $R=\{\sigma_m:m\ge 0\}$. By parts \ref{p3} and \ref{p4}, 
\begin{equation}\label{rest}
    \ord(\sigma)<\omega^2\text{ for all }\sigma\not\in R.
\end{equation}

Given $\sigma_k$. If there is $\sigma\not \in R$ such that $\sigma_k\succsim \sigma$, then 
$$\ord(\sigma_k)\le \ord(\sigma)<\omega^2<\omega^2+k.$$
Hence, we obtain the upper bound in \eqref{ok} for this case.

Consider the remaining case when $\sigma\succ \sigma_k$ for all $\sigma\not\in R$.
Suppose $\ord(\sigma_k)>\omega^2+k$.
By \eqref{orule} applied to $\sigma_*=\sigma_k$ and \eqref{rest}, we must have 
$$\ord(\{\sigma_0,\ldots,\sigma_{k-1}\})\supset \{\omega^2,\omega^2+1,\ldots,\omega^2+k\}.$$ 
Comparing the numbers of elements in these two sets gives $k\ge k+1$, which is impossible.
Therefore, we  must have $\ord(\sigma_k)\le \omega^2+k$, that is,  the upper bound of $\ord(\sigma_k)$ in \eqref{ok}.
\end{proof}

The following remarks on case \ref{p2} of Theorem \ref{thm1} are in order.
\begin{enumerate}
    \item Taking the scalar product of \eqref{rel1} in $H$ with $w_1$  gives $\langle B(w_1,w_1),v\rangle=0$, that is,
    $B(w_1,w_1)$ and $v$ are orthogonal in $H$.
    
    \item If, in addition, $w_2$ exists, then 
    instead of the first inequality in \eqref{os2} one has
       $$3\in\{\ord(\sigma_0),\ord(\sigma_{1,1}),\ord(\sigma_{0,2})\},$$
    and instead of \eqref{os1} one has
       $$3\in\{\ord(\sigma_1),\ord(\sigma_{1,1}),\ord(\sigma_{0,2})\}.$$
These can easily be seen from \eqref{seqarray}.
\end{enumerate}
\subsection{Proof of Proposition \ref{v0eg}}\label{pf32}

\begin{proof} We have $g\ne 0$, so $v_n\ne 0=v$. Hence, expansion \eqref{infser} is nontrivial. Consequently, $w_1$ exists and $|w_1|=1$.

Using $v_n=\Gamma_{1,n}w_n^{(1)}$, we have from \eqref{steady} that
\begin{equation}\label{t1}
    \Gamma_{1,n}Aw_n^{(1)}+\alpha_n \Gamma_{1,n}^2B(w_n^{(1)},w_n^{(1)})=g.
\end{equation}

If $\sigma_0\succ \sigma_{1,1}$, then passing $n\to\infty$ in \eqref{t1} gives $g=0$, a contradiction. Thus, we must have 
$\sigma_{1,1}\succsim \sigma_0$.

\medskip
We now prove $w_2$ exists by contradiction.
Suppose expansion \eqref{infser} stops at $w_1$, that is $v_n=\Gamma_{1,n}w_1$. Then
\begin{equation}\label{t2}
    \Gamma_{1,n}Aw_1+\alpha_n \Gamma_{1,n}^2B(w_1,w_1)=g.
\end{equation}

\medskip\noindent
\emph{Case A. $\sigma_{1,1}\succ \sigma_0$.} Dividing  \eqref{t2}  by $\alpha_n \Gamma_{1,n}^2$ and then passing $n\to\infty$ yield $B(w_1,w_1)=0$.
Using this fact in \eqref{t2} implies $\Gamma_{1,n}Aw_1=g$, which yields $g=0$, a contradiction.

\medskip\noindent
\emph{Case B. $\sigma_{1,1}\sim \sigma_0$.} 
By Lemma \ref{concat} and using a subsequential set of $\mathcal S$, but still denoted by $\mathcal S$, we assume \eqref{xn} and three possibilities (S1), (S2), (S3).
Passing $n\to\infty$ in \eqref{t2} gives 
\begin{equation}\label{t4}
    \lambda_* B(w_1,w_1)=g.
\end{equation} 

We can rewrite \eqref{t2} as
\begin{equation}\label{t3}
        \Gamma_{1,n}Aw_1=\chi_n g.
\end{equation}

\begin{enumerate}
    \item[($\alpha$)]\label{aa}  If $\chi_n=0$ for all $n$, or  $\Gamma_{1,n}\succ |\chi_n|$, then we infer from \eqref{t3} that $Aw_1=0$, which is a contradiction.
    \item[($\beta$)]\label{bb} If $|\chi_n|\succ \Gamma_{1,n}$, then $g=0$, a contradiction.
    \item[($\gamma$)]\label{gg} If $\Gamma_{1,n}\sim |\chi_n|$, then $\lambda_1 Aw_1=g$ for some $\lambda_1\ne 0$.
Combining this with \eqref{t4} gives
\begin{equation*}
    \lambda_1 Aw_1=\lambda_* B(w_1,w_1).
\end{equation*}
Taking scalar product of the last equation with $w_1$, we deduce $w_1=0$, a contradiction.
\end{enumerate}

\noindent 
Since both cases A and B yield contradictions, the expansion \eqref{infser} cannot stop at $w_1$. Thus, $w_2$ exists.

\medskip
Now, we can substitute the expansion $v_n=\Gamma_{1,n}w_1+\Gamma_{2,n}w_n^{(2)}$ into the steady state equation \eqref{steady}  to obtain
\begin{multline} \label{z2m}
\Gamma_{1,n}Aw_1+\Gamma_{2,n}Aw_n^{(2)} +\alpha_n\Gamma_{1,n}^2 B(w_1,w_1)
+\alpha_n \Gamma_{1,n}\Gamma_{2,n}\Bs(w_1,w_n^{(2)})\\
 +\alpha_n \Gamma_{2,n}^2B(w_n^{(2)},w_n^{(2)})=g.
\end{multline}
\begin{enumerate}[label=\tnum]
    \item \textit{Case $\sigma_{1,1} \succ \sigma_0$.} Since $\alpha_n\Gamma_{1,n}^2\to\infty$ we have also $\alpha_n\Gamma_{1,n}\to\infty$. Dividing \eqref{t1} by $\alpha_n\Gamma_{1,n}^2$ and take $n\to\infty$, we obtain 
$ B(w_1,w_1)=0$. 

If $1 \succ \alpha_n \Gamma_{1,n}\Gamma_{2,n}$, then all terms on the left in \eqref{z2m} tend to 0 as $n\to \infty$, which would mean $g=0$, a contradiction. Hence, $\sigma_{1,2}\succsim  \sigma_0$, and consequently, 
$\alpha_n \Gamma_{2,n}\to\infty$.

Having eliminating the term $\alpha_n\Gamma_{1,n}^2 B(w_1,w_1)$ from \eqref{z2m}, items \ref{f1} and \ref{f2} follow from dividing the equation by $\alpha_n \Gamma_{1,n}\Gamma_{2,n}$, passing $n\to\infty$ and applying both relations in \eqref{flip}. 

\item \textit{Case $\sigma_{1,1} \sim \sigma_0$.} We take $n \to \infty$ in \eqref{t1} and use the limit in \eqref{xn} to obtain  \eqref{t4}.

Using identity \eqref{t4}, we rewrite \eqref{z2m} as
\begin{align} \label{z3m}
\Gamma_{1,n}Aw_1+\Gamma_{2,n}Aw_n^{(2)}
+\alpha_n \Gamma_{1,n}\Gamma_{2,n}\Bs(w_1,w_n^{(2)})
+\alpha_n \Gamma_{2,n}^2B(w_n^{(2)},w_n^{(2)})=\chi_n g.
\end{align}
If $1\succ \alpha_n \Gamma_{2,n}$, then $\Gamma_{1,n}\succ \alpha_n \Gamma_{1,n}\Gamma_{2,n}\succ \alpha_n \Gamma_{2,n}^2$.  We next eliminate the three possibilities that follow.
\begin{enumerate}
 \item[($\alpha'$)]\label{aap}   If $\chi_n=0$ for all $n$,  or $\Gamma_{1,n}\succ |\chi_n|$, then \eqref{z3m} implies  $Aw_1=0$.
 
    \item[($\beta'$)]\label{bbp} If $|\chi_n|\succ \Gamma_{1,n}$, then \eqref{z3m} implies $g=0.$

    \item[($\gamma'$)]\label{ggp}  If $|\chi_n|\sim \Gamma_{1,n}$, then \eqref{z3m} implies 
$g=\lambda_1 Aw_1$, with $\lambda_1\ne 0.$
\end{enumerate}

Similar to ($\alpha$), ($\beta$), ($\gamma$) above, none of these cases  ($\alpha'$), ($\beta'$), ($\gamma'$)  are possible. 
Therefore, we must have $\alpha_n \Gamma_{2,n}\succsim 1$, that is, $\sigma_{0,2}\succsim \sigma_0$. 
It remains to consider the following cases.

\begin{enumerate}[label=\nnum]
    \item \emph{Case} $\sigma_{0,2}\sim \sigma_0$. Then $\Gamma_{1,n}\sim \alpha_n\Gamma_{1,n}\Gamma_{2,n}\succ \alpha_n \Gamma_{2,n}^2$.  If $|\chi_n|\succ \Gamma_{1,n}$, then \eqref{z3m} implies $g=0$, a contradiction. 
\begin{enumerate}[label=\rnum]
        \item If $|\chi_n|\sim \Gamma_{1,n}$, then, since 
 $w_n^{(2)}\to w_2$, equation \eqref{z3m} implies 
$$g=\lambda_1 Aw_1 +\lambda_2 \Bs(w_1,w_2),\quad \lambda_1\lambda_2> 0.$$
Thus
$$\lambda_* B(w_1,w_1)=\lambda_1 Aw_1 +\lambda_2 \Bs(w_1,w_2).$$
Taking scalar product with $w_1$ gives
$$\lambda_1\|w_1\|^2=\lambda_2 \langle B(w_1,w_1),w_2\rangle 
=(\lambda_2/\lambda_*) \langle g,w_2\rangle.$$
       \item If $\chi_n=0$ for all $n$, or $\Gamma_{1,n}\succ |\chi_n|$, then \eqref{z3m} implies 
 $$Aw_1 + \lambda_2 \Bs(w_1,w_2)=0,\quad \lambda_2> 0,$$
which yields 
$$ \|w_1\|^2 -\lambda_2 \langle B(w_1,w_1),w_2\rangle =0 \quad \text{and hence} \quad
(\lambda_2/\lambda_*) \langle g,w_2\rangle=\|w_1\|^2.$$    
    \end{enumerate}
    \item  \emph{Case} $\sigma_{0,2}\succ \sigma_0$. Then $\alpha_n \Gamma_{1,n}\Gamma_{2,n}\succ \Gamma_{1,n}$.  If $|\chi_n|\succ \alpha_n \Gamma_{1,n}\Gamma_{2,n}$, then \eqref{z3m} implies  $g=0$ again. 

\begin{enumerate}[label=\rnum]
    \item If $|\chi_n|\sim \alpha_n \Gamma_{1,n}\Gamma_{2,n}$. 
Then \eqref{z3m} yields $g=\lambda_2 \Bs(w_1,w_2)$ for some $\lambda_2\ne 0.$
Combining this with \eqref{t4} gives
\begin{equation}\label{z4m}
    \lambda_* B(w_1,w_1)=\lambda_2 \Bs(w_1,w_2).
\end{equation}

On the one hand, taking scalar product of equation \eqref{z4m} with $w_1$ yields
 $\langle B(w_1,w_1),w_2\rangle=0$, hence, $\langle g,w_2\rangle=0$.

On the other hand, taking scalar product of equation \eqref{z4m} with $w_2$ gives $\langle B(w_2,w_2),w_1\rangle=0$.

   \item  If $\chi_n=0$ for all $n$, or $\alpha_n \Gamma_{1,n}\Gamma_{2,n} \succ |\chi_n|$, then \eqref{z3m} gives
 \begin{equation}\label{z5m}
 \Bs(w_1,w_2)=0.     
 \end{equation}
Then, again, taking scalar product of \eqref{z5m} with $w_1$ yields $\langle g,w_2\rangle=0$, and taking scalar product of \eqref{z5m}  with $w_2$ gives $\langle B(w_2,w_2),w_1\rangle=0$.
\end{enumerate}
\end{enumerate}
\end{enumerate}
\end{proof}
\subsection{Proof of Proposition \ref{vAinvg}}\label{pf33}
\begin{proof} 
By \eqref{Bzero} in Theorem \ref{thm1}, we have $B(v,v)=0$.
 Substituting $v_n=v+\Gamma_{1,n}w_n^{(1)}$ into the steady state equation \eqref{steady}, we have
\begin{equation}\label{t5}
\Gamma_{1,n}Aw_n^{(1)}
+\alpha_n \Gamma_{1,n} \Bs(v,w_n^{(1)})  +\alpha_n \Gamma_{1,n}^2B(w_n^{(1)},w_n^{(1)})=0,    
\end{equation}
The largest term $\alpha_n \Gamma_{1,n} $ gives \eqref{B1}.

Suppose expansion \eqref{infser} stops at $w_1$, that is, $v_n=v+\Gamma_{1,n}w_1$.
We have \eqref{t5} with $w_1$ replacing $w_n^{(1)}$, and with the use of \eqref{B1}, we obtain
\begin{equation}\label{t6}
       Aw_1+ \alpha_n \Gamma_{1,n}B(w_1,w_1)=0.
\end{equation}

\begin{itemize}
    \item If $1\succ \alpha_n \Gamma_{1,n}$, then \eqref{t6} implies $Aw_1=0$, a contradiction.

    \item If $\alpha_n \Gamma_{1,n}\succ 1$, then \eqref{t6} implies $B(w_1,w_1)=0$, which, thanks to \eqref{t6} again, yields $Aw_1=0$, a contradiction.

    \item If $1\sim \alpha_n \Gamma_{1,n}$, then \eqref{t6} implies 
$$Aw_1+\lambda B(w_1,w_1)=0\text{ for some number }\lambda>0.$$
Taking the scalar product of this equation with $w_1$ yields $w_1=0$, a contradiction.
\end{itemize}

In conclusion, expansion \eqref{infser} cannot stop at $w_1$, hence $w_2$ exists.

Next we insert  $v_n=v+\Gamma_{1,n}w_1+\Gamma_{2,n}w_n^{(2)}$ into the state equation to obtain
\begin{equation}\label{t7}
\begin{aligned}
&    \Gamma_{1,n}Aw_1+\Gamma_{2,n}Aw_n^{(2)}
+\alpha_n \Gamma_{2,n} \Bs(v,w_n^{(2)}) 
+\alpha_n \Gamma_{1,n}^2B(w_1,w_1)\\
&
+\alpha_n \Gamma_{1,n}\Gamma_{2,n} \Bs(w_1,w_n^{(2)}) 
+\alpha_n \Gamma_{2,n}^2B(w_n^{(2)},w_n^{(2)})=0.  
\end{aligned}
\end{equation} 

For \ref{m1},  it follows \eqref{t7} that
\begin{enumerate}
    \item [] if $\Gamma_{1,n}\succ \alpha_n \Gamma_{2,n},\alpha_n \Gamma_{1,n}^2$, then $Aw_1=0$, so we immediately have $w_1=0$,
    \item [] if $\Gamma_{1,n}\sim \alpha_n \Gamma_{1,n}^2 \succ \alpha_n \Gamma_{2,n}$, then 
    $$Aw_1 +\lambda B(w_1,w_1)=0\quad \text{for some} \quad \lambda> 0.$$
    Thus taking the scalar product with $w_1$ and applying \eqref{flip}, we find that, again $w_1=0$.
\end{enumerate}

Therefore, these two cases are impossible. 

The results for the remaining items \ref{m2}--\ref{m6} follow by considering the largest sequence among $\Gamma_{1,n}$, $\alpha_n \Gamma_{2,n} $ and $\alpha_n \Gamma_{1,n}^2$ in equation \eqref{t7}.
\end{proof}

\section{Expanded force case} \label{expandsec}
In this section, with positive numbers $\alpha_n$'s as in Section \ref{results}, we consider a family of forces $(g_n)_{n=1}^\infty$ in $\mathcal H$ and a family of corresponding steady states $(v_n)_{n=1}^\infty\subset \mathcal H$ of the Galerkin approximation, i.e.,
\begin{equation}\label{eqforce}
Av_n+\alpha_n B(v_n,v_n)=g_n,
\end{equation}
where, we recall, $B$ is defined by \eqref{Bbar}.
We assume that $g_n$ is bounded in $\mathcal H$, so that $v_n$ is as well. 
Using Lemma \ref{mainlem} we can find subsequences of $v_n$'s and $g_n$'s, still denoted by $v_n$ and $g_n$, with the strict unitary expansions
\begin{equation}\label{genevg}
v_n\approx v+\sum_k \Gamma_{k,n}w_k,\quad
g_n\approx g+\sum_k  H_{k,n} h_k.
\end{equation}

Let $\mathcal S$ be defined as in section \ref{results}.
Denote 
\begin{equation*}
    \beta_k=(H_{k,n})_{n=1}^\infty\in\mathcal X\text{ and }
    \widetilde{\mathcal S}=\mathcal S \cup \{ \beta_k \text{ with valid integers $k$'s}\}.
\end{equation*}
Since we work with subsequences, by the countability of $\widetilde{\mathcal S}$ and the virtue of Lemma \ref{total} and \eqref{vsubsum}, we may assume that $\widetilde{\mathcal S}$ is totally comparable. Then we can apply the idea of the procedure \eqref{GP} in Section \ref{results} to equation \eqref{eqforce} and sequences in $\widetilde{\mathcal S}$.

\begin{thm}\label{gnthm} Assume $g\ne 0$ and the strict unitary expansion of $g_n$ in \eqref{genevg} is nontrivial. Then one has the following.
    \begin{enumerate}[label=\tnum]
        \item $B(v,v)=0$.
        
        \item The strict unitary expansion of $v_n$ in \eqref{genevg} is nontrivial.

        \item If $\alpha_n \Gamma_{1,n}\succ 1$, then 
        \begin{equation}\label{g63}
        \Bs(v,w_1)=0.    
        \end{equation}

       \item If $1\succ \alpha_n \Gamma_{1,n}$, then $\alpha_n \Gamma_{1,n}\succsim H_{1,n}$ and 
        \begin{align}\label{g64}
        &Av=g,\\   
        \label{g65}
        &\Bs(v,w_1)=\lambda_1  h_1\text{ where  }\lambda_1=\lim_{n\to\infty} \frac{H_{1,n}}{\alpha_n \Gamma_{1,n}}\in[0,\infty).
        \end{align}
        
        \item If $\alpha_n \Gamma_{1,n}\sim 1$, then
        \begin{equation}\label{g66}
            Av-g+\lambda_2 \Bs(v,w_1)=0, \text{ where }  \lambda_2=\lim_{n\to\infty} \alpha_n \Gamma_{1,n}>0.
        \end{equation}

        Let $\chi_n=\alpha_n \Gamma_{1,n}-\lambda_1$. Assume (S1)--(S3) in Section \ref{results} with $\widetilde{\mathcal S}$ replacing $\mathcal S$.
        \begin{enumerate}[label=\nnum]
            \item  If 
            \begin{equation}\label{gc1}
                \Gamma_{1,n}\succ H_{1,n}\text{ and }\lim_{n\to\infty} \frac{\chi_n}{\Gamma_{1,n}}=0,
            \end{equation}
            then $w_2$ exists and $\alpha_n \Gamma_{2,n}\succsim \Gamma_{1,n}$ and 
                \begin{equation}\label{g67}
                \lambda_3(Aw_1+\lambda_2 B(w_1,w_1))+\Bs(v,w_2)=0,\text{ where }\lambda_3=\lim_{n\to\infty} \frac{\Gamma_{1,n}}{\alpha_n \Gamma_{2,n}}\in[0,\infty).
                \end{equation}
                
            \item  If 
            \begin{equation}\label{gc2}
                H_{1,n} \succ \Gamma_{1,n}\text{ and }\lim_{n\to\infty} \frac{\chi_n}{H_{1,n}}=0,
            \end{equation} 
            then $w_2$ exists and $\alpha_n \Gamma_{2,n}\succsim H_{1,n}$ and 
                \begin{equation}\label{g68}
                    \Bs(v,w_2)=\lambda_4 h_1,\text{ where }\lambda_4=\lim_{n\to\infty} \frac{H_{1,n}}{\alpha_n \Gamma_{2,n}}\in [0,\infty).
                \end{equation}

            \item If $w_2$ exists and 
            \begin{equation}\label{gc3}
                \alpha_n\Gamma_{2,n}\succ \Gamma_{1,n},H_{1,n}\text{ and }\lim_{n\to\infty} \frac{\chi_n}{\alpha_n\Gamma_{2,n}}=0,
            \end{equation} 
            then $\Bs(v,w_2)=0$. 
            
            \item Assume $|\chi_n|\succ \Gamma_{1,n}, H_{1,n}$.
                If $w_2$ does not exist, then \eqref{g63} and \eqref{g64} hold.
                If $w_2$ exists and $|\chi_n| \succsim \alpha_n\Gamma_{2,n}$, then 
                \begin{equation}\label{g69}
                    \Bs(v,w_1)+\lambda_5 \Bs(v,w_2)=0,\text{ where } \lambda_5=\lim_{n\to\infty} \frac{\alpha_n\Gamma_{2,n}}{\chi_n}\in \mathbb R.
                \end{equation}

            \item Assume $\Gamma_{1,n} \sim H_{1,n}\sim |\chi_n|$. If $w_2$ does  not exist then
            \begin{equation}\label{g70}
                Aw_1+\lambda_6 \Bs(v,w_1)+\lambda_2 B(w_1,w_1)=\lambda_7 h_1,
            \end{equation}
            where 
            $$\lambda_6=\lim_{n\to\infty} \frac{\chi_n}{\Gamma_{1,n}}\in\mathbb R,
            \quad \lambda_7=\lim_{n\to\infty} \frac{H_{1,n}}{\Gamma_{1,n}}>0.$$
            
            If $w_2$ exists and $\Gamma_{1,n} \succsim \alpha_n\Gamma_{2,n}$, then 
            \begin{equation}\label{g71}
                Aw_1+\lambda_6 \Bs(v,w_1)+\lambda_2 B(w_1,w_1)+\lambda_8 \Bs(v,w_2)=\lambda_7 h_1,
            \end{equation}
             where 
            $$\lambda_8=\lim_{n\to\infty} \frac{\alpha_n \Gamma_{2,n}}{\Gamma_{1,n}}\in[0,\infty).$$
        \end{enumerate}
    \end{enumerate}
\end{thm}
\begin{proof}
We write
    \begin{equation}\label{gp1}
        g_n=g+H_{1,n}h_n^{(1)}\text{ with } \lim_{n\to\infty} h_n^{(1)}=h_1\ne 0.
    \end{equation}

\medskip\noindent
Part (i).     Dividing \eqref{eqforce} by $\alpha_n$ and letting $n\to\infty$ give $B(v,v)=0$.

\medskip\noindent
Part (ii). Suppose $v_n=v$ for all $n$. Then, together with part (i), we have from \eqref{eqforce} that  $Av-g=H_{1,n}h_{n}^{(1)}$.
This yields $Av-g=0=H_{1,n}h_n^{(1)}$. The last identity contradicts the nonzero limit in \eqref{gp1}.

\medskip\noindent
Part (iii). Assume $\alpha_n \Gamma_{1,n}\succ 1$.
Substituting $v_n=v+\Gamma_{1,n}w_n^{(1)}$ and  \eqref{gp1} into \eqref{eqforce} and using part (i) give
\begin{equation}\label{gp2}
    (Av-g)+\Gamma_{1,n}Aw_n^{(1)} + \alpha_n \Gamma_{1,n} \Bs(v,w_n^{(1)} )+\alpha_n \Gamma_{1,n}^2 B(w_n^{(1)} ,w_n^{(1)} )
    =H_{1,n}h_n^{(1)}.
\end{equation}
Dividing this equation by $\alpha_n\Gamma_{1,n}$ and letting $n\to\infty$, we obtain \eqref{g63}.

\medskip\noindent
Part (iv). Assume $1\succ \alpha_n \Gamma_{1,n}$. Letting $n\to\infty$ in \eqref{gp2} yields \eqref{g64}.
Now, \eqref{gp2} becomes
\begin{equation}\label{gp3}
\Gamma_{1,n}Aw_n^{(1)} + \alpha_n \Gamma_{1,n} \Bs(v,w_n^{(1)} )+\alpha_n \Gamma_{1,n}^2 B(w_n^{(1)} ,w_n^{(1)} )
    =H_{1,n}h_n^{(1)}.
\end{equation}

Suppose $H_{1,n}\succ \alpha_n \Gamma_{1,n}$. 
By dividing \eqref{gp3} by $H_{1,n}$ and letting $n\to\infty$, we obtain $h_1=0$, which is a contradiction. Thus $\alpha_n \Gamma_{1,n}\succsim H_{1,n}$.  Dividing \eqref{gp3} by $\alpha_n \Gamma_{1,n}$ and letting $n\to\infty$, we obtain \eqref{g63}.

\medskip\noindent
Part (v). Assume $\alpha_n \Gamma_{1,n}\sim 1$. Passing $n\to\infty$ in \eqref{gp2} gives \eqref{g65} directly.

(1) Assume \eqref{gc1}. Suppose $w_n^{(1)}=w_1$ for all $n$. We rewrite \eqref{gp2} using \eqref{g66} as
\begin{equation}\label{gp4}
    \Gamma_{1,n} Aw_1+\chi_n \Bs(v,w_1)+\alpha_n \Gamma_{1,n}^2 B(w_1,w_1)=H_{1,n} h_n^{(1)}.
\end{equation}

Note that $\Gamma_{1,n}\sim \alpha_n \gamma_{1,n}^2$. 
Dividing \eqref{gp4} by $\Gamma_{1,n}$, passing $n\to\infty$ and using condition \eqref{gc1}, we obtain 
\begin{equation}\label{gp5}
    Aw_1+\lambda_1 \Bs(w_1,w_1)=0.
\end{equation}
This implies $w_1=0$, a contradiction. 
Therefore, $w_2$ exists. We rewrite \eqref{gp2} using \eqref{g66} and the fact $\Gamma_1 w_n^{(1)}=\Gamma_{1,n} w_1+\Gamma_{2,n} w_n^{(2)}$ as
\begin{equation}\label{gp6}
\begin{aligned}
&    \Gamma_{1,n}Aw_1+\Gamma_{2,n}Aw_n^{(2)}
+\chi_n \Bs(v,w_1) 
+\alpha_n \Gamma_{2,n} \Bs(v,w_n^{(2)}) 
+\alpha_n \Gamma_{1,n}^2B(w_1,w_1)\\
&
+\alpha_n \Gamma_{1,n}\Gamma_{2,n} \Bs(w_1,w_n^{(2)}) 
+\alpha_n \Gamma_{2,n}^2B(w_n^{(2)},w_n^{(2)})=H_{1,n}h_n^{(1)}.  
\end{aligned}
\end{equation} 

We focus on the terms with coefficients $\Gamma_{1,n}$, $\chi_n $, $\alpha_n \Gamma_{2,n} $, $H_{1,n}$ in \eqref{gp6}.

If $\Gamma_{1,n}\succ \alpha_n \Gamma_{2,n}$, then dividing \eqref{gp6} by $\Gamma_{1,n}$ and letting $n\to\infty$, we obtain \eqref{gp5} again which yields a contradiction. Thus, $\alpha_n \Gamma_{2,n}\succsim \Gamma_{1,n}$.
Now,  dividing \eqref{gp6} by $\alpha_n \Gamma_{2,n}$ and letting $n\to\infty$, we obtain \eqref{g67}. 

(2) Assume \eqref{gc2}. If $w_n^{(1)}=w_1$ for all $n$, then dividing \eqref{gp4} by $H_{1,n}$ and passing $n\to\infty$ give $h_1=0$, a contradiction. Thus $w_2$ exists. Similar to the rest of part (1), by dividing \eqref{gp6} by $H_{1,n}$ we derive 
$\alpha_n \Gamma_{2,n}\succsim H_{1,n}$, and then by dividing \eqref{gp6} by $\alpha_n \Gamma_{2,n}$ we obtain \eqref{g68}.

(3) Assume $w_2$ exists and \eqref{gc3}. Then dividing \eqref{gp6} by $\alpha_n \Gamma_{2,n}$ and letting $n\to\infty$, we obtain $\Bs(v,w_2)=0$.

(4) Assume $|\chi_n|\succ \Gamma_{1,n}, H_{1,n}$.
                If $w_2$ does not exists, then \eqref{gp4} implies \eqref{g63}. This and \eqref{g66} then imply \eqref{g64}.
If $w_2$ exists and $|\chi_n| \succsim \alpha_n\Gamma_{2,n}$, then dividing \eqref{gp5} by $\chi_n$ and letting $n\to\infty$ yield \eqref{g69}.

(5) Similar to part (4), dividing \eqref{gp4} and \eqref{gp5} by $\Gamma_{1,n}$ yield \eqref{g70} and \eqref{g71}, respectively.
\end{proof}

We observe that the procedure \eqref{GP} in Section \ref{results}, in fact, does not require $|w_k|=|w_n^{(k)}|=1$. 
This prompts us to consider expansions of a more general class than those satisfying Definition \ref{uexp}.

\begin{defn}\label{uexp2}
Let $(v_n)_{n=1}^\infty$ be a sequence in a normed space $(Z,\|\cdot\|_Z)$ over $\mathbb C$ or $\mathbb R$. 
\begin{itemize}
    \item We say the sequence $(v_n)_{n=1}^\infty$ has a \emph{unitary expansion} if, in Definition \ref{uexp}, condition $\|w_n^{(k)}\|_Z=1$ is removed.

    \item We say the sequence $(v_n)_{n=1}^\infty$  has a \emph{pre-unitary expansion} if in Definition \ref{uexp}, condition $\|w_k\|_Z=1$ is replaced with $w_k\ne 0$, and condition $\|w_n^{(k)}\|_Z=1$ is removed.  
\end{itemize}
\end{defn}

We still use \eqref{vsum}  to denote the asymptotic expansions in Definition \ref{uexp2}. 

\begin{obs}\label{propec}
Suppose $(v_n)_{n=N_0}^\infty$ has a pre-unitary expansion \eqref{vsum}.
Then it has the following  unitary expansion
\begin{equation}\label{vconvert}
    v_n\approx v+\sum_k \widehat\Gamma_{k,n}\widehat w_k,\text{ where } 
    \widehat w_k=\|w_k\|_Z^{-1} w_k,\ \widehat\Gamma_{k,n}=\|w_k\|_Z \Gamma_{k,n}.
\end{equation}

Indeed, let $\widehat w_n^{(k)} =\|w_k\|_Z^{-1} w_n^{(k)}$.
Clearly, $\|\widehat w_k\|_Z=1$,  $\widehat \Gamma_{k,n}\to 0$ and  
$$\widehat\Gamma_{k+1,n}/\widehat\Gamma_{k,n}=(\|w_{k+1}\|_Z/\|w_k\|_Z) \Gamma_{k+1,n}/\Gamma_{k,n}\to 0\text{ as }n\to\infty.$$
Moreover, one has, for $n\ge N_0$ and $k\ge 1$, that $\Gamma_{k,n} w_k=\widehat \Gamma_{k,n}\widehat w_k$ and $\Gamma_{k,n} w_n^{(k)}=\widehat \Gamma_{k,n}\widehat w_n^{(k)}$.
    For $n\ge N_0$ and $m\ge 1$, we write
    \[v_n=v+\sum_{k=1}^{m-1} \Gamma_{k,n} w_k+\Gamma_{k,n} w_n^{(m)}
    =v+\sum_{k=1}^{m-1} \widehat\Gamma_{k,n}\widehat w_k+\widehat\Gamma_{k,n} \widehat w_n^{(m)}.\]
Since $\widehat w_n^{(m)}$ converges to $\|w_m\|_Z^{-1}w_m =\widehat w_m$  as $n\to\infty$,
we have the unitary expansion  \eqref{vconvert}.

Therefore, any pre-unitary expansion can be converted to a unitary one by a (natural) scaling specified in \eqref{vconvert}.
\end{obs}

\begin{obs}\label{finrmk} We have the following remarks on finite unitary expansions.
\begin{enumerate}[label=\rnum]
    \item 
Suppose $(v_n)_{n=N_0}^\infty$ has a finite unitary expansion. Then taking $k=K$ in \eqref{b4} and using \eqref{b6}, we obtain
\begin{equation}\label{finsum}
    v_n=v+\Gamma_{1,n}w_1+\ldots +\Gamma_{K,n}w_K, \text{ for all } n\ge N_0,
\end{equation}

Now, suppose we have \eqref{finsum} 
with $\Gamma_{j,n}$ and $w_j$, for $1\le j\le K$, being as in Definition \ref{uexp}.     
Then the sum in \eqref{finsum} is a unitary expansion of $(v_n)_{n=N_0}^\infty$. Indeed, we have \eqref{b4} with
$$w_n^{(k)}=w_k+\sum_{j=k+1}^K \frac{\Gamma_{j,n}}{\Gamma_{k,n}}w_j.$$
Clearly, $w_n^{(k)}\to w_k$ as $n\to\infty$.

\item In the case of a finite unitary expansion \eqref{finsum} in a finite dimensional real linear space $Z$, by taking $N_0$ sufficiently large, we may assume $K\le {\rm dim}(Z)$.
Indeed, assume we have \eqref{finsum} for some $K>{\rm dim}(Z)$.
Suppose $w_1,\ldots, w_k$ are linearly independent and $w_{k+1}\in \text{span}\{ w_1,\ldots, w_k\}$. Then$$
 w_{k+1}=c_1 w_1+\cdots + c_k w_k \text{ for some }c_1,c_2,\ldots,c_k\in\mathbb R.
 $$
Setting $\tilde\Gamma_{j,n}=\Gamma_{j,n}+c_j\Gamma_{k+1,n}$, for $1\le j\le k$, we have
\begin{align*}
 \Gamma_{1,n}w_1+ \cdots + \Gamma_{k,n}w_k+\Gamma_{k+1,n}w_{k+1}
&= \tilde\Gamma_{1,n}w_1+ \cdots + \tilde \Gamma_{k,n}w_k
\end{align*}
For large enough $n$, we have $\tilde\Gamma_{j,n} > 0$, and  
$$\tilde\Gamma_{j+1,n}/\tilde\Gamma_{j,n} \to 0 \quad \text{as} \ n \to \infty,
\quad \text{for} \ j=1,\ldots k-1. $$

Repeating the above argument finitely many times starting with $k=1$, we obtain \eqref{finsum} with newly defined $\Gamma_{j,n}$, re-indexed $w_j$ and new $K\le {\rm dim}(Z)$.
\end{enumerate}
\end{obs}

A simple example for unitary expansions is the power series in Banach spaces, namely,
\[ v_n=v+\sum_{k=1}^\infty \lambda^{kn} w_k,\text{ for some number $\lambda\in(0,1)$ and unit vectors $w_k$'s.}\] 

This can be generalized as follow.

\begin{lem}\label{convlem}
Let $\theta_n>0$ and $\theta_n\to 0$ as $n\to\infty$.
Suppose $(w_n)_{n=0}^\infty$ is a sequence in a Banach space $(Z,\|\cdot\|_Z)$ over $\mathbb C$ or $\mathbb R$  and there exist numbers $M>0$ and $D_0>0$ such that  
\begin{equation*}
    \|w_n\|_Z\le MD_0^n\text{ for all }n\ge 1.
\end{equation*}

Then there is $N_0\ge 1$ such that $\sum_{k=0}^\infty \theta_n^{k}w_k$ converges absolutely to a vector $v_n\in Z$, for all $n\ge N_0$. Moreover, one has, for all $n\ge N_0$ and $m\ge 1$ that 
\begin{equation}\label{vnthew}
    v_n=\sum_{k=0}^{m-1} \theta_n^{k}w_k+\theta_n^{m}w_n^{(m)} \text{ with } \lim_{n\to\infty}w_n^{(m)}=w_m.
\end{equation} 
\end{lem}
\begin{proof}
 Let $\gamma\in(0,1)$ and integer $N_0\ge 1$ be such that $D_0\theta_n\le \gamma$ for all $n\ge N_0$. 
We have
\begin{equation}
    \theta_n^{k}\|w_k\|_Z\le M(D_0\theta_n)^k\le M\gamma^k\text{ for }n\ge N_0.
\end{equation}
This implies $\sum_{k=0}^\infty \theta_n^{k}w_k$ converges absolutely to a vector $v_n\in Z$ for any $n\ge N_0$.

For $n\ge N_0$ and $m\ge 1$, we write $v_n$ as in \eqref{vnthew} with $w_n^{(m)}=w_m+ \sum_{k=m+1}^\infty \theta_n^{k-m}w_k$.

We estimate the last sum by 
   \begin{align*}
       \sum_{k=m+1}^\infty \theta_n^{k-m}\|w_k\|_Z
       &\le M \theta_n^{-m} \sum_{k=m+1}^\infty (D_0\theta_n)^{k}= M \theta_n^{-m} \frac{(D_0\theta_n)^{m+1}}{1-D_0\theta_n}
       =\frac{MD_0^{m+1}\theta_n}{(1-D_0\theta_n)},
   \end{align*}
which goes to $0$ as $n\to\infty$.
Thus, $\lim_{n\to\infty} w_n^{(m)}=w_m$.
\end{proof}

Returning to \eqref{genevg} and taking into account Lemma \ref{convlem}, we look for the following specific expansions.
For $n\ge 1$, let
\begin{equation}\label{vngn}
 v_n=\sum_{k=0}^\infty \alpha_n^{-k} w_k\text{ and }
  g_n=\sum_{k=0}^\infty \alpha_n^{-k} h_k. 
\end{equation}

Below, we use heuristic arguments to find the relations between $w_k$'s and $h_k$'s.
Substituting expansions in \eqref{vngn} into \eqref{eqforce} and collecting the  $\alpha_n^{-m}$-terms, for $m\ge -1$, we formally derive, for $m=-1$, 
\begin{equation}\label{Bw0}
B(w_0,w_0)=0,
\end{equation}
and, for $m\ge 0$, 
\begin{equation}\label{hm}
 Aw_m + \sum_{k=0}^{m+1} B(w_k,w_{m+1-k})=h_m,
\end{equation}
i.e.,
\begin{equation}\label{hmBs}
 Aw_m + \Bs(w_0,w_{m+1})+\sum_{k=1}^m B(w_k,w_{m+1-k})=h_m.
\end{equation}

In particular, when $m=0$,
\begin{equation}\label{h0}
 Aw_0 + \Bs(w_0,w_1)=h_0,
\end{equation}
and when $m=1$,
\begin{equation}\label{h1}
 Aw_1 + \Bs(w_0,w_2)+B(w_1,w_1)=h_1.
\end{equation}

We will show that the above constructions can be made rigorous.
Let 
$$M_A=\max\{|Av|: v\in \mathcal H,|v|=1\} \text{ and } M_B=\max\{|B(u,v)|: u,v\in\mathcal H,|u|=|v|=1\}.$$
Then 
\begin{equation}\label{ABineq}
|v|\le     |Av|\le M_A |v|,\quad |B(u,v)|\le M_B  |u|\, |v| \text{ for all } u,v\in \mathcal H.
\end{equation}

\begin{defn}
For $u\in \mathcal H$, we define a linear mapping $L_u$  on $\mathcal H$ by
$$ w\in\mathcal H\mapsto L_{u}w := Aw + \Bs(u,w)\in\mathcal H.$$
\end{defn}

Denote $c_0=1/(4(M_B+1))>0$. Let  $u\in \mathcal H$ with $|u|\le c_0$. 
By \eqref{ABineq}, one has $$|L_u w|\ge |w|-2M_B |u||w|\ge |w|/2.$$
This implies $L_u$ is invertible, and 
\begin{equation}\label{Luinv}
    |L_u^{-1}f|\le 2|f|\text{ for all }f\in \mathcal H.
\end{equation}

\begin{thm}\label{nonzthm}
Given $w_0\in \mathcal H$ that satisfies \eqref{Bw0}. Then there exist $h_0\in\mathcal H$ and $w_k,h_k\in \mathcal H\setminus\{0\}$ for all $k\ge 1$, and an integer $N_0\ge 1$ such that \eqref{hm} holds for all $m\ge 0$, $v_n$ and $g_n$ defined by \eqref{vngn} are absolutely convergent for all $n\ge N_0$ and satisfy equation \eqref{eqforce} for all $n\ge N_0$. 
Moreover, the series in \eqref{vngn} are pre-unitary expansions of $(v_n)_{n=N_0}^\infty$ and $(g_n)_{n=N_0}^\infty$, that is,
\begin{equation}\label{vgpre}
    v_n\approx w_0+\sum_{k=1}^\infty \Gamma_{k,n}w_k\text{ and } g_n\approx  h_0+\sum_{k=1}^\infty \Gamma_{k,n}h_k \text{ with } \Gamma_{k,n}=\alpha_n^{-k}.
\end{equation}
\end{thm}
\begin{proof}
First of all, we define $h_0$ by \eqref{h0}.
Given any numbers $M\ge 1$ and $D_0>1$.

\medskip\noindent
Step 1. We will find $w_k,h_k\in \mathcal H\setminus\{0\}$, for all $k\ge 1$, that satisfy \eqref{hm} and 
\begin{equation}\label{wcond2}
    |w_k|\le M D_0^k \text{ for all } k\ge 1.
\end{equation}

We consider two cases.

\medskip
\emph{Case 1. The mapping $u\in\mathcal H\mapsto \Bs(w_0,u)$ is not zero.}
We construct $w_k$'s and $h_k$'s recursively.
First, choose $w_1\in \mathcal H$ such that $0<|w_1|\le M D_0$ and define $h_0$ by \eqref{h0}.

Let $m\ge 1$. Suppose we already have $w_1,\ldots,w_{m}$ and $h_1,\ldots,h_{m-1}$. 

(a) If $Aw_m+\sum_{k=1}^{m}B(w_k,w_{m+1-k})=0$, choose $w_{m+1}\in \mathcal H\setminus\{0\}$ such that $|w_{m+1}|\le M D_0^{m+1}$ and 
$ \Bs(w_0,w_{m+1})\ne 0.$

(b) If $Aw_m+\sum_{k=1}^{m}B(w_k,w_{m+1-k})\ne 0$, choose $w_{m+1}\in \mathcal H\setminus\{0\}$ so that $|w_{m+1}|\le M D_0^{m+1}$  and
$$ |\Bs(w_0,w_{m+1})|<\left|Aw_m+\sum_{k=1}^{m}B(w_k,w_{m+1-k})\right|. $$

For both cases (a) and (b), define $h_m$ by \eqref{hmBs}. Then $h_m\ne 0$. 

\medskip
\emph{Case 2. $\Bs(w_0,u)=0$ for all $u\in\mathcal H$.}
Equation \eqref{hmBs} for $m\ge 1$ now becomes 
\begin{align}
\label{sh1}
&Aw_1+B(w_1,w_1)=h_1, \text{ see \eqref{h1} when $m=1$,}
\end{align}
and, for $m\ge 2$,
\begin{align}
\notag 
&Aw_m+\Bs(w_1,w_m)+\sum_{2\le k\le m-1}  B(w_k,w_{m+1-k})=h_m\\
\intertext{which is equivalent to}
\label{shm}
&L_{w_1}w_m = f_m:=h_m-\sum_{2\le k\le m-1}  B(w_k,w_{m+1-k}).
\end{align}

We choose $h_1\in \mathcal H$ such that  $0<|h_1|\le c_0$.
We solve for $w_1$ from the Galerkin NSE \eqref{sh1} with the body force $h_1$.
Taking scalar product (in $\mathcal H$) of \eqref{sh1} with $w_1$, we have
$$|w_1|^2\le \langle Aw_1,w_1\rangle =\langle h_1,w_1\rangle\le \frac12|w_1|^2+\frac12|h_1|^2.$$
Hence, $0<|w_1|\le |h_1|\le c_0.$ Since $c_0<1$, we have $|w_1|\le M D_0$.

Let $m\ge 2$. Suppose we already have $h_k,w_k$ for $k=1,\ldots,m-1$.
We choose $h_m\in \mathcal H\setminus\{0\}$ such that the $f_m$ defined in \eqref{shm} satisfies $0<|f_m|\le M D_0^m/2$.
Choose $w_m=L_{w_1}^{-1}f_m$. Then $w_m\ne 0$ and, by inequality \eqref{Luinv},  
$|w_m|\le 2|f_m|\le M D_0^m.$

\medskip\noindent
Step 2. For $m\ge 1$,  we have from \eqref{hm}, \eqref{ABineq} and \eqref{wcond2} that
\begin{align*}
|h_m|&\le M_A |w_m|+\sum_{k=0}^{m+1}M_B |w_k||w_{m+1-k}|
\le M_A M D_0^m+(m+2)M_B M^2 D_0^{m+1}.
\end{align*}

This implies that $\limsup_{m\to\infty} |h_m|^{1/m}<\infty$.
We apply Lemma \ref{convlem} to both $w_n$ and $h_n$ using $\theta_n=\alpha_n^{-1}$. Then there exists $N_0\ge 1$ such that the series in \eqref{vngn}
converge absolutely for any $n\ge N_0$, and they are the pre-unitary expansions \eqref{vgpre} of $(v_n)_{n=N_0}^\infty$ and $(g_n)_{n=N_0}^\infty$.

\medskip\noindent
Step 3.
By the absolute convergence of series for $v_n$ and (similar) Cauchy's product for $B(v_n,v_n)$, one can prove, with the use of relations \eqref{Bw0} and \eqref{hm}, that $v_n$ and $g_n$ satisfy equation \eqref{steady} for all $n\ge N_0$. 
\end{proof}

\begin{obs}\label{matrmk}
Thanks to Remark \ref{propec},  we can convert the pre-unitary expansions in \eqref{vgpre} to the following unitary ones
\begin{align}
\label{vhat}
    v_n&\approx w_0+\sum_{k=1}^\infty \widehat\Gamma_{k,n}\widehat w_k,\text{ where } 
    \widehat w_k=|w_k|^{-1} w_k,\ \widehat\Gamma_{k,n}=|w_k| \alpha_n^{-k},\\
\label{ghat}
    g_n&\approx h_0+\sum_{k=1}^\infty \widehat H_{k,n}\widehat h_k,\text{ where } 
    \widehat h_k=|h_k|^{-1} h_k,\ \widehat H_{k,n}=|h_k| \alpha_n^{-k}.
\end{align}

We have
$1\sim \alpha_n \widehat\Gamma_{1,n}$ and $\widehat H_{k,n} \sim  \widehat\Gamma_{k,n} \sim \alpha_n  \widehat\Gamma_{j,n} \widehat\Gamma_{k+1-j,n}$  for $k\ge 1$, $1\le j\le k$.

Substituting \eqref{vhat} and \eqref{ghat} into equation \eqref{eqforce} and following the procedure \eqref{GP} in Section \ref{results} with the use of equivalent sequences, we obtain, for each $m\ge 0$,
$$  
\widehat\Gamma_{m,n} A\widehat w_m + \alpha_n \sum_{k=0}^{m+1} \widehat\Gamma_{k,n} \widehat\Gamma_{m+1-k,n} B(\widehat w_k, \widehat w_{m+1-k})
= \widehat H_{m,n} \widehat h_m
$$
which is exactly equivalent to \eqref{hm}. This shows that the procedure \eqref{GP} works precisely in this case.
\end{obs}

The next Example \ref{Bszero} and Proposition \ref{Bsnonz} show that both Cases 1 and 2 in the proof of Theorem \ref{nonzthm} are probable.

Consider the periodic boundary condition. 
For  $\mathbf k\in \mathbb Z^d\setminus\{0\}$, denote 
$E_{\mathbf k}(\bx)=e^{i\mathbf k\cdot \bx}$ for $\bx\in\mathbb R^d$. 
For any nonempty, finite subset  $K$ of $\mathbb Z^d\setminus\{0\}$ that is symmetric about the origin, we denote
$$E[K]=\{u\in H: u=\sum_{\mathbf k\in K} \hat \bu_{\bk} E_{\mathbf k}(\bx)\}.$$

We say $\mathcal H$ covers a wave vector $\mathbf k\in \mathbb Z^d\setminus\{0\}$ if 
$E[\{\mathbf k,-\mathbf k\}] \subset \mathcal H$.
Clearly, $\mathcal H$ covers $\mathbf k$ if and only if it covers $(-\mathbf k)$.

\begin{eg}\label{Bszero}
Although Case 2 in the proof of Theorem \ref{nonzthm} looks peculiar, it still holds true for some particular $v$ and  space $\mathcal H$ which can have an arbitrary large dimension.
Here is an example. Consider the periodic boundary condition. Let $M\ge 1$ be a real number, and $\mathbf k$ be a given vector in $\mathbb Z^d$ such that $|\mathbf k|>2M$.

Set $K_1=\{\mathbf k,-\mathbf k\}$ and $K_2=\{\mathbf j\in\mathbb Z^d:0<|\mathbf j|\le M\}$.
Let $K_3$ be a finite subset of $\{\mathbf p \in\mathbb Z^d\setminus K_2:\mathbf p\ne 0, \mathbf p\cdot \mathbf k=0\}$ such that $K_3$ is symmetric about the origin.

Define $K=K_1\cup K_2\cup K_3$ and let $\mathcal H=E[K]$.
Let $v$ be any function in $E[K_1]\subset \mathcal H$.  

Let $u$ be any function in $\mathcal H$.
Then $u=u_1+u_2+u_3$, where $u_j\in E[K_j]$ for $j=1,2,3$. 
We have
\[ \Bs(v,u)=\Bs(v,u_1)+\Bs(v,u_2)+\Bs(v,u_3).\]

Obviously, $\Bs(v,u_1)=0.$

Consider $\Bs(v,u_2)$. 
Observe that the functions $(v\cdot \nabla )u_2$ and $(u_2\cdot\nabla)v$ are complex linear combinations of functions  $E_{\mathbf j\pm \mathbf k}(\bx)$ with $\mathbf j\in K_2$.
For $\mathbf j\in K_2$, one has $|\mathbf j\pm \mathbf k|\ge |\mathbf k|-|\mathbf j|>M$. 
Thus $\mathbf j\pm \mathbf k\not\in K_2$. 
Also, $(\mathbf k\pm \mathbf j)\cdot \mathbf k\ge |\mathbf k|^2-|\mathbf j||\mathbf k|\ge 2M|\mathbf k|-M|\mathbf k|>0$, hence $\mathbf j\pm \mathbf k\not \in K_3$.
It follows that $\mathbf j\pm \mathbf k\not\in K$ or $\mathbf j\pm \mathbf k\in K_1$.
With both of these cases, we have, after projection to $\mathcal H$,  
$B(v,u_2)=B(u_2,v)=0$.

Consider $\Bs(v,u_3)$. Let $\mathbf p\in K_3$.
Then, by the orthogonality,  $|\mathbf p\pm \mathbf k|>|k|>2M$, which implies $\mathbf p\pm \mathbf k\not\in K_2$.
Also, $(\mathbf p\pm \mathbf k)\cdot \mathbf k=\pm|\mathbf k|^2\ne 0$, hence $\mathbf p\pm \mathbf k\not \in K_3$.
Then, similar to the consideration of $\Bs(v,w_2)$, we must have $B(v,u_3)=B(u_3,v)=0$.

In conclusion, we have $\Bs(v,u)=0$. Finally, the large dimension of $\mathcal H$ comes from the large size of $K$.
This, in turn, comes from the large size of $K_3$ regardless of $M$ and/or the fact that large number $M$ gives a large size of $K_2$.
\end{eg}

When $\mathcal H=\bar P_\Lambda H$, Proposition \ref{Bsnonz} below is a general result for Case 1 in the proof of Theorem \ref{nonzthm}. 

\begin{lem}\label{B2Dlem}
    Consider the 2D case with the periodic boundary condition. Let $\mathbf k,\mathbf j\in \mathbb Z^2\setminus\{0\}$ be such that $\mathbf k,\mathbf j$ are not parallel and $|\mathbf k|\ne |\mathbf j|$. Suppose $\mathcal H$ covers $\mathbf k$, $\mathbf j$, $\mathbf k+\mathbf j$.
    Assume $v$ is a nonzero function in $E[\{\mathbf k,-\mathbf k\}]$ and $w$ is a nonzero function in $E[\{\mathbf j,-\mathbf j\}]$. Then 
    the coefficient of $E_{\mathbf k+\mathbf j}(\bx)$ in $\Bs(v,w)$ is nonzero, and, consequently, $\Bs(v,w)\ne 0$.
\end{lem}
\begin{proof}
There are $z_{\mathbf k},\zeta_{\mathbf j}\in\mathbb C\setminus\{0\}$ such that
$$v=z_{\mathbf k} (\mathbf e_3\times \mathbf k) E_{\mathbf k}(\bx)+\bar z_{\mathbf k} (\mathbf e_3\times \mathbf k) E_{-\mathbf k}(\bx),\ 
w=\zeta_{\mathbf j} (\mathbf e_3\times \mathbf j) E_{\mathbf j}(\bx)+\bar \zeta_{\mathbf j} (\mathbf e_3\times \mathbf j) E_{-\mathbf j}(\bx). $$

We have
$$
(v\cdot\nabla )w+(w\cdot \nabla )v=(\mathbf a E_{\mathbf k+\mathbf j}(\bx)+\bar{\mathbf a} E_{-\mathbf k-\mathbf j}(\bx))
+(\mathbf b E_{\mathbf k-\mathbf j}(\bx)+\bar {\mathbf b} E_{\mathbf j-\mathbf k}(\bx)),
$$
with $\mathbf a,\mathbf b\in\mathbb C^2$.
Clearly, all vectors $\mathbf k+\mathbf j$, $-\mathbf k-\mathbf j$, $\mathbf k-\mathbf j$, $\mathbf j-\mathbf k$ are nonzero and mutually distinct. Hence, the coefficient of $E_{\mathbf k+\mathbf j}(\bx)$ in $(v\cdot\nabla )w+(w\cdot \nabla )v$ is $\mathbf a$, which explicitly is
\begin{equation}\label{tri0}
\mathbf a= iz_{\mathbf k}\zeta_{\mathbf j}((\mathbf e_3\times \mathbf k)\cdot \mathbf j)\mathbf e_3\times \mathbf j+ iz_{\mathbf k}\zeta_{\mathbf j}((\mathbf e_3\times \mathbf j)\cdot \mathbf k)\mathbf e_3\times\mathbf  k ,    
\end{equation}
or equivalently, 
\begin{equation*}
    \mathbf a=iz_{\mathbf k}\zeta_{\mathbf j}((\mathbf e_3\times \mathbf k)\cdot \mathbf j)\mathbf e_3\times (\mathbf j-\mathbf k).
\end{equation*}
Therefore, the coefficient of $E_{\mathbf k+\mathbf j}(\bx)$ in $\Bs(v,w)$ is 
$$\mathbf c=\mathbf a-\frac{\mathbf a\cdot(\mathbf k+\mathbf j)}{|\mathbf k+\mathbf j|^2} (\mathbf k+\mathbf j).$$ 

Note that
$\mathbf c\times  (\mathbf j+\mathbf k)=iz_{\mathbf k}\zeta_{\mathbf j}((\mathbf e_3\times \mathbf k)\cdot \mathbf j)
[\mathbf e_3\times (\mathbf j-\mathbf k)]\times  (\mathbf j+\mathbf k)$.

Because $\mathbb Z^2$-vectors $\mathbf k,\mathbf j$ are not parallel, we have $(\mathbf e_3\times \mathbf k)\cdot \mathbf j\ne 0$. Also,
\begin{equation*}
[\mathbf e_3\times (\mathbf j-\mathbf k)]\times (\mathbf j+\mathbf k) = -(\mathbf j-\mathbf k)\cdot(\mathbf j+\mathbf k) \mathbf e_3=(|\mathbf k|^2-|\mathbf j|^2)\mathbf e_3\ne 0.
\end{equation*}
Therefore, $\mathbf c\ne 0$, which implies  $\Bs(v,w)\ne 0$.
\end{proof}

\begin{prop}\label{Bsnonz}
Consider the periodic boundary condition. 
Let $\Lambda\ge 5$ be an eigeinvalue of $A$. Assume $\mathcal H=\bar P_\Lambda H$.
Let $v$ be any nonzero function in $\bar P_{(\sqrt \Lambda-1)^2}H$.
Then there exists $w\in\mathcal H$ such that
$\Bs(v,w)\ne 0$.
\end{prop}
\begin{proof}
In this case, we have 
$\mathcal H=E[\{\mathbf k\in\mathbb Z^d:0<|\mathbf k|^2\le \Lambda\}]$ and
$$v\in \bar P_{(\sqrt \Lambda-1)^2}H=E[\{\mathbf k\in\mathbb Z^d: 0<|\mathbf k|\le \sqrt \Lambda-1\}]\subset \mathcal H.$$ 
We use lexicographic order for the wave vectors in the finite Fourier series of $v$.

\medskip\noindent\emph{The 2D case.} 
Let $K$ be the set of wave vectors in the finite Fourier series of $v$ with nozero coefficients. 
Let $\mathbf k=\max K$, $K_0=\{\mathbf k,-\mathbf k\}$ and $K_1=K\setminus K_0$.
We write $v=v_0+v_1$ where $v_0\ne 0$, $v_0\in E[K_0]$ and $v_1\in E[K_1]$.

We will choose $w$ to be a nonzero function in $E[\{\mathbf k',-\mathbf k'\}]$ for a chosen wave vector $\mathbf k'$ having nonnegative coordinates and at least one positive coordinate. We have 
\begin{equation*}
    \Bs(v,w)=\Bs(v_0,w)+\Bs(v_1,w).
\end{equation*}

Consider the coefficient of function $E_{\mathbf k+\mathbf k'}(\bx)$ in $\Bs(v,w)$.
Observe that $\Bs(v_1,w)$ is a finite sum of functions of modes $\mathbf j\pm \mathbf k$ with $\mathbf j\in K_1$.
Let $\mathbf j$ be any vector in $K_1$. Then $\mathbf j<\mathbf k$ which yields $\mathbf j+\mathbf k'\ne \mathbf k+\mathbf k'$. Suppose $\mathbf j-\mathbf k'=\mathbf k+\mathbf k'$. Then $\mathbf j=\mathbf k+2\mathbf k'$. Since $\mathbf k'>(0,0)$, this implies $\mathbf j>\mathbf k$, a contradiction. 
Thus, $\mathbf j-\mathbf k'\ne \mathbf k+\mathbf k'$.
We conclude that $\Bs(v_1,w)$  does not contribute to the coefficient of $E_{\mathbf k+\mathbf k'}(\bx)$.
Therefore, 
\begin{equation}\label{reduce}
\text{coefficient of $E_{\mathbf k+\mathbf k'}(\bx)$ in $\Bs(v,w)$ is exactly that in $\Bs(v_0,w)$.}    
\end{equation}

Suppose $\mathbf k=(k_1,k_2)$. Then $k_1\ge 0$.
\begin{itemize}
    \item If $k_1\ge 2$ choose $\mathbf k'=(0,1)$.

    \item If $k_1=1$, $|k_2|\ge 1$ choose $\mathbf k'=(0,1)$.

    \item If $k_1=1$, $k_2=0$ choose $\mathbf k'=(0,2)$. Then $|\mathbf k+\mathbf k'|^2=5$.

    \item If $k_1=0$, $|k_2|\ge 2$ choose $\mathbf k'=(1,0)$.

    \item If $k_1=0$, $|k_2|=1$ choose $\mathbf k'=(2,0)$. Then $|\mathbf k+\mathbf k'|^2=5$.
\end{itemize}

Note that $|\mathbf k|^2<\Lambda$ and $|\mathbf k'|^2\le 4$.
Also, except for the case $|\mathbf k+\mathbf k'|^2=5$, the remaining cases have $|\mathbf k'|=1$. The last case implies 
$|\mathbf k+\mathbf k'|^2\le (|\mathbf k|+1)^2\le \Lambda$.
Therefore, $\mathcal H$ covers $\mathbf k$, $\mathbf k'$ and $\mathbf k+\mathbf k'$.
Observe that $\mathbf k,\mathbf k'$ are not parallel and $|\mathbf k|\ne |\mathbf k'|$.
By the virtue of Lemma \ref{B2Dlem}, the coefficient of $E_{\mathbf k+\mathbf k'}(\bx)$ in $\Bs(v_0,w)$ is nonzero.
Combining this with \eqref{reduce}, we obtain $\Bs(v,w)\ne 0$.

\medskip\noindent\emph{The 3D case.} 
We proceed as in the 2D case up to \eqref{reduce}.
In this case, 
\begin{equation}\label{vzform}
    v_0=\mathbf a_{\mathbf k} E_{\mathbf k}(\bx)+\bar {\mathbf a}_{\mathbf k} E_{-\mathbf k}(\bx)\text{ with }0\ne \mathbf a_{\mathbf k}\in \mathbb C^3,\ 
\mathbf a_{\mathbf k}\cdot \mathbf k=0.
\end{equation}

We choose $w=(\mathbf k\times \mathbf k')(E_{\mathbf k'}(\bx)+E_{-\mathbf  k'}(\bx)),$
where $\mathbf k'\in \mathbb Z^3$ additionally satisfies it is not parallel with $\mathbf k$.
Then, similar to \eqref{tri0}, the coefficient of $E_{\mathbf k+\mathbf k'}(\bx)$ in $(v_0\cdot\nabla)w+(w\cdot \nabla )v_0$ is
\begin{align*}
    i (\mathbf a_{\mathbf k}\cdot \mathbf k')(\mathbf k\times \mathbf k')+i ((\mathbf k\times \mathbf k')\cdot \mathbf k)\mathbf a_{\mathbf k}
    =i (\mathbf a_{\mathbf k}\cdot \mathbf k')(\mathbf k\times \mathbf k').
\end{align*}

Since $\mathbf k'$ is not parallel with $\mathbf k$, we have $\mathbf k\times \mathbf k'$ and $\mathbf k+\mathbf k'$ are nonzero and perpendicular to each other. Hence $(\mathbf k\times \mathbf k')$ is not parallel with $(\mathbf k+\mathbf k')$.

We will select $\mathbf k'$ so that  $\mathbf a_{\mathbf k}\cdot \mathbf k'\ne 0$. We do so for all possible scenarios of $\mathbf k$.
\begin{itemize}
\item Case $\mathbf k$ is parallel with $\mathbf e_1$. By the last condition in \eqref{vzform}, $\mathbf a_{\mathbf k}$ is in $\mathbf e_2,\mathbf e_3$-plane. So $\mathbf a_{\mathbf k}\cdot \mathbf e_2\ne 0$ or $\mathbf a_{\mathbf k}\cdot \mathbf e_3\ne 0$. 
We choose $\mathbf k'=\mathbf e_2$ or $\mathbf k'=\mathbf e_3$ correspondingly. 

\item Case $\mathbf k$ is parallel with $\mathbf e_2$. Choose  $\mathbf k'=\mathbf e_1$ when  $\mathbf a_{\mathbf k}\cdot \mathbf e_1\ne 0$, or $\mathbf k'=\mathbf e_3$ when  $\mathbf a_{\mathbf k}\cdot \mathbf e_3\ne 0$.

\item Case $\mathbf k$ is parallel with $\mathbf e_3$. Choose  $\mathbf k'=\mathbf e_1$ when  $\mathbf a_{\mathbf k}\cdot \mathbf e_1\ne 0$, or $\mathbf k'=\mathbf e_2$ when  $\mathbf a_{\mathbf k}\cdot \mathbf e_2\ne 0$.

\item Case $\mathbf k$ is not parallel with $\mathbf e_1,\mathbf e_2,\mathbf e_3$. 
Then either $\mathbf a_{\mathbf k}\cdot \mathbf e_1\ne 0$, or $\mathbf a_{\mathbf k}\cdot \mathbf e_2\ne 0$, or $\mathbf a_{\mathbf k}\cdot \mathbf e_3\ne 0$.
Then choose $\mathbf k'=\mathbf e_1$, or  $\mathbf k'=\mathbf e_2$ or $\mathbf k'=\mathbf e_3$, correspondingly.
\end{itemize}

Since $|\mathbf k'|=1$, we have, same as in the 2D case, $\mathcal H$ covers $\mathbf k$, $\mathbf k'$ and $\mathbf k+\mathbf k'$.
Then the coefficient of $E_{\mathbf k+\mathbf k'}(\bx)$ in  $\Bs(v_0,w)$ is not zero.
By \eqref{reduce}, we have $\Bs(v,w)\ne 0$.
\end{proof}

With the expansions in \eqref{vngn}, we have $g_n\to g=h_0$ and $v_n\to v=w_0$. If $v=0$, then by \eqref{h0},  we must have $g=0$. However, as we will see in the next result, with a different design of expansions, it is possible that $v_n\to 0$ but $g_n\to g\ne 0$.

\begin{thm}\label{vzthm2}
Suppose there exists $u\in \mathcal H$ such that 
\begin{equation}\label{Bnz}
B(u,u)\ne 0   . 
\end{equation} 

For any $M>0$, there are $g_n\in \mathcal H$ and solutions $v_n$ of \eqref{eqforce}, for all $n\ge 1$,  such that  $g_n\to g$ and $v_n\to 0$, as $n\to\infty$, with $|g|=M$.  
\end{thm}
\begin{proof}
Take 
\begin{equation}\label{vnstop}
    v_n=\alpha_n^{-1/2}w_1\text{ and }g_n=g+\alpha_n^{-1/2}h_1,
\end{equation}
where 
$$w_1=\frac{\sqrt{M}}{\sqrt{|B(u,u)|}}u, \quad 
g=B(w_1,w_1),\quad h_1=Aw_1.$$
Direct calculations show that the conclusions of this theorem hold. 
\end{proof}

\begin{eg}\label{Dex} 
Consider the periodic boundary condition. 
Clearly, condition \eqref{Bnz} is not met if $\mathcal H=E[K]$ with all wave vectors in $K$ being parallel to each other.
We give explicit examples below where \eqref{Bnz} does hold.

\medskip
\emph{The 2D case.} Let $\mathbf k=2\mathbf e_1$ and $\mathbf j=\mathbf e_2$. Assume $\mathcal H$ covers $\mathbf k,\mathbf j,\mathbf k+\mathbf j$. Set $u=u_1+u_2$, where
$$u_1=(\mathbf e_3\times \mathbf k)(E_{\mathbf k}(\bx)+E_{-\mathbf k}(\bx))\text{ and }
u_2=(\mathbf e_3\times \mathbf j)(E_{\mathbf j}(\bx)+E_{-\mathbf j}(\bx)).$$
Then $B(u,u)=\Bs(u_{1},u_{2})$. By Lemma \ref{B2Dlem}, $\Bs(u_{1},u_{2})\ne 0$ which implies \eqref{Bnz}.

\medskip
\emph{The 3D case.} Suppose $\mathcal H$ covers $\mathbf e_1$, $\mathbf e_2$ and $\mathbf e_1+\mathbf e_2$.
Let $$u=u_{1}+u_{2},\text{ with }
u_{1}=\mathbf e_2 (E_{\mathbf e_1}(\bx)+E_{-\mathbf e_1}(\bx)),\
u_{2}= \mathbf e_3 (E_{\mathbf e_2}(\bx)+E_{-\mathbf e_2}(\bx)).
$$
Then
$B(u,u)=\Bs(u_{1},u_{2}).$
The coefficient of $E_{\mathbf e_1+\mathbf e_2}(\bx)$ in $(u_1\cdot\nabla) u_2 +(u_2\cdot\nabla )u_1$ is
$$(\mathbf e_2\cdot i\mathbf e_2)\mathbf e_3+ (\mathbf e_3\cdot i\mathbf e_1)\mathbf e_2=i\mathbf e_3.$$
This vector is not parallel with $\mathbf e_1+\mathbf e_2$, hence $\Bs(u_{1},u_{2})\ne 0$, and we have \eqref{Bnz}.
\end{eg}

\begin{obs}\label{gcompare}
Recall from Proposition \ref{v0eg} for the case $g_n\equiv g$ that the expansion of $v_n$ must contain the $w_2$-term. This contrasts with the fact the expansion of $v_n$ in \eqref{vnstop} stops at $w_1$.
\end{obs}

\appendix

\section{}\label{ap}

The proof of Lemma \ref{total} needs the following preparations in Lemmas \ref{lemA1}--\ref{lemA3}.

\begin{lem}\label{lemA1}
 For any $x,y\in \mathcal X$, the  set $\{x,y\}$ has a totally comparable subsequential set.

\end{lem}
\begin{proof}
    Suppose $x=(x_n)_{n=1}^\infty$ and $y=(y_n)_{n=1}^\infty$. If $x_n/y_n$ is unbounded, then there is a subsequence $(n_k)_{k=1}^\infty$ of $(n)_{n=1}^\infty$ such that $ (x_{n_k})_{k=1}^\infty\succ (y_{n_k})_{k=1}^\infty$. Otherwise, there is a subsequence $(n_k)_{k=1}^\infty$ such that
    $x_{n_k}/y_{n_k}\to \lambda\in[0,\infty)$ as $k\to\infty$. If $\lambda>0$ then $ (x_{n_k})_{k=1}^\infty\sim (y_{n_k})_{k=1}^\infty$, otherwise
    $ (y_{n_k})_{k=1}^\infty\succ (x_{n_k})_{k=1}^\infty$.
\end{proof}

\begin{lem}\label{lemA2}
    Let $X$ be a totally comparable, finite subset of $\mathcal X$ and $y\in \mathcal X$. Then $X\cup\{y\}$ has a totally comparable subsequential set.
\end{lem}
\begin{proof}
By finite induction on the size of $X$ and the use of Lemma \ref{lemA1}.    
\end{proof}

\begin{lem}\label{lemA3}
    Let $X$ be a finite, non-empty subset of $\mathcal X$. Then it has a totally comparable subsequential set.
\end{lem}
\begin{proof}
Suppose $X=\{x_1,x_2,\ldots,x_N\}$ with $x_k=(x_{k,n})_{n=1}^\infty$. For $1\le k\le N$, let $Y_k=\{x_1,x_2,\ldots,x_k\}$.
Then $Y_{k+1}=Y_k\cup \{x_{k+1}\}$ for $1\le k\le N-1$. Applying Lemma \ref{lemA2} repeatedly, we obtain subsequnces $(\varphi_k(n))_{n=1}^\infty$ such that $\varphi_1(n)=n$, $(\varphi_{k+1}(n))_{n=1}^\infty$ is a subsequence of $(\varphi_k(n))_{n=1}^\infty$, and $(Y_{k+1})_{\varphi_k}=(Y_k)_{\varphi_k}\cup \{(x_{k+1,\varphi_k(n)})_{n=1}^\infty\}$ has a totally comparable subsequential set $(Y_{k+1})_{\varphi_{k+1}}$.
Therefore, $X_{\varphi_N}=(Y_N)_{\varphi_N}$ is totally comparable.
\end{proof}

\begin{proof}[Proof of Lemma \ref{total}]
The case $X$ is finite was already proved in Lemma \ref{lemA3}. Consider $X$ is infinite now.
We proceed as in the proof of Lemma \ref{lemA3}. Suppose $X=\{x_k=(x_{k,n})_{n=1}^\infty:k\ge 1\}$ and let $Y_k=\{x_1,x_2,\ldots,x_k\}$ for $k\ge 1$. Then there exist subsequnces $(\varphi_k(n))_{n=1}^\infty$ of $(n)_{n=1}^\infty$, for all $k\ge 1$, such that  $(\varphi_{k+1}(n))_{n=1}^\infty$ is a subsequence of $(\varphi_k(n))_{n=1}^\infty$, and each $(Y_{k})_{\varphi_k}$ is totally comparable.
By setting $\varphi(n)=\varphi_n(n)$, we have $X_\varphi$ is a subsequential set of $X$ and it is totally comparable.
\end{proof}

For $X,Y\subset \mathcal X$, we say $X\succ Y$ if $x\succ y$ for all $x\in X$ and $y\in Y$.
Its negation is denoted by $X\not \succ Y$.

\begin{proof}[Proof of the second inequality of \eqref{ojk} in Theorem \ref{thm1}, part \ref{p4}]
Let $R$ be the first row in \eqref{seqarray}, that is,
 $R=\{\sigma_m:m\ge 0 \}$.
We state a refined version, namely, for any $j\ge 1$ and $k\ge j$,
\begin{equation}\label{neword}
    \ord(\sigma_{j,k})\le \begin{cases}
        \omega\cdot (2j)&\text{ if } R\not\succ \{\sigma_{j,k}\},\\
        \omega\cdot (2j)+ (k-j)(k-j+1)/2  &\text{ if } R\succ \{\sigma_{j,k}\}.         
    \end{cases}
\end{equation}

Below, we prove \eqref{neword} by induction in $j$.

\emph{Step 1. Consider $j=1$.} We will prove, for all $k\ge 1$, that
\begin{equation}\label{s1}
    \ord(\sigma_{1,k})\le \begin{cases}
        \omega+\omega&\text{ if } R\not\succ \{\sigma_{1,k}\},\\
        \omega+\omega + (k-1)k/2  &\text{ if } R\succ \{\sigma_{1,k}\}.         
    \end{cases}
\end{equation}

Given $k\ge 1$. We need to compare $\sigma_{1,k}$ with $R$.

\medskip\noindent
\emph{Case 1. $R \not \succ \{\sigma_{1,k}\}$, that is, there is $m$ such that $ \sigma_{1,k} \succsim \sigma_m $.} Suppose
\begin{equation}\label{ns1}
\ord(\sigma_{1,k}) > \omega+\omega.
\end{equation}

Let $T$ be the triangle of sequences in \eqref{seqarray} with vertices $\sigma_{1,1}$, $\sigma_{1,k-1}$, and $\sigma_{k-1,k-1}$. Explicitly,
    $T=\{\sigma_{j,p}:1\le j\le p\le k-1\}.$

Thanks to part \ref{p3}, we observe that the possible entries $\sigma$ that satisfy $\sigma\succ \sigma_{1,k}$ and $\ord(\sigma)\ge \omega$ are $\sigma_0,\sigma_1,\ldots,\sigma_{m-1}$ along with the triangle $T$.  
From this fact, property \eqref{orule} applied to $\sigma_*=\sigma_{1,k}$ and \eqref{ns1},  we must have
\begin{equation}\label{ordrel}
\begin{aligned}
&\ord(\{\sigma_0,\ldots,\sigma_{m-1} \}\cup T)
\supset \{\text{ordinal number } \zeta: \omega \le \zeta<\ord(\sigma_{1,k})\}\\
&\supset \{\text{ordinal number } \zeta: \omega \le \zeta\le \omega+\omega\}
= \{\omega,\omega+1,\ldots,\omega+\omega\}.
\end{aligned}
\end{equation}

Since the first set is finite and the last set is infinite, the inclusion is impossible.

\medskip\noindent
\emph{Case 2. $R\succ \{\sigma_{1,k}\}$, that is,  $\sigma_m\succ \sigma_{1,k} $ for all $m$.}
Assume the contrary of the second inequality in \eqref{s1}, that is,
\begin{equation}\label{ns2}
\ord(\sigma_{1,k}) > \omega+\omega+(k-1)k/2.
\end{equation}

In this case, the possible entries $\sigma$ that satisfy $\sigma\succ \sigma_{1,k}$ and $\ord(\sigma)\ge \omega$ are those in the row $R$ and the triangle $T$.
Combining this fact with \eqref{orule} applied to $\sigma_*=\sigma_{1,k}$ and \eqref{ns2},  we must have, similar to \eqref{ordrel}, that
\begin{equation}\label{ST}
\ord(R\cup T) \supset \{\omega,\omega+1,\ldots,\omega+\omega,\omega+\omega+1,\omega+\omega+2,\ldots,\omega+\omega+(k-1)k/2\}.
\end{equation}

We study the order in $R\cup T$ next.
We compare $R$ and $T$.
Decompose $T=T_1\cup T_2$, where 
$$T_2=\{\sigma\in T: \sigma_m\succ\sigma\text{ for all } m\}
\text{ and }T_1=T \setminus T_2.
$$

Note that all three sets $T,T_1,T_2$ are finite and $R\succ T_2$, which clearly implies $R\cap T_2=\emptyset$.

For $\sigma\in T_1$, there is the smallest number $m(\sigma)$ such that $\sigma\succsim \sigma_{m(\sigma)}$.
Let $$m_*=\max\{m(\sigma):\sigma\in T_1\}.$$
Then for $\sigma\in T_1$ we have $\sigma\succsim  \sigma_{m(\sigma)}\succsim  \sigma_{m_*}$.
The set $R\cup T_1$ is partitioned into $R_1\cup R_2$ with
\begin{equation}\label{RR}
R_1=\{\sigma_0,\ldots, \sigma_{m_*}\}\cup T_1,
\quad 
R_2=\{\sigma_{m_*+p}:p\ge 1\},
\end{equation}

We have $R_1\succ R_2\succ T_2$ and $R\cup T=R_1\cup R_2\cup T_2$.

From \eqref{ST} and \eqref{RR}, the finite set $R_1$ can only cover the assigned ordinal numbers going from $\omega$ to $\omega+M$ with some  finite number $M$. 
The infinite set $R_2$ will give ordinal numbers $$\omega+M+1, \omega+M+2, \omega+M+3, \ldots, \text{ i.e., }\omega+M+p, \text{ for all integers }p\ge 1.$$
Thus, the finite set $T_2$ has assigned ordinal numbers $\omega+\omega, \omega+\omega+1,\ldots, \omega+\omega+K$, for some nonnegative integer $K$. Therefore, the maximum of $\ord(R\cup T)$ is $\omega+\omega+K$.
Note that $K+1$ cannot exceed the number of elements in $T_2$, which cannot exceed the number of elements in $T$ 
which is $(k-1)k/2$. Thus, $K<(k-1)k/2$, which gives a contradiction to \eqref{ST}.
Therefore, we obtain the second inequality in \eqref{s1}.

\medskip
From both cases 1 and 2 above, one has that inequality \eqref{s1} holds.

\emph{Induction step.} Let $j\ge 2$. Assume, for all $1\le j'\le j-1$ and $k'\ge  j'$,  that
\begin{equation}\label{sprime}
    \ord(\sigma_{j',k'})\le \begin{cases}
        \omega\cdot (2j')&\text{ if } R\not\succ \{\sigma_{j',k'}\},\\
        \omega\cdot (2j')+ (k'-j')(k'-j'+1)/2  &\text{ if } R\succ \{\sigma_{j',k'}\}.          
    \end{cases}
\end{equation}

As a consequence of hypothesis \eqref{sprime}, we have 
\begin{equation}\label{rough}
     \ord(\sigma_{j',k'})<\omega\cdot (2(j-1))+\omega=\omega\cdot(2j-1) \text{ for all $1\le j'\le j-1$ and $k'\ge  j'$.}
\end{equation}
That is all elements $\sigma$ on the rows starting with $\sigma_{j',j'}$ for $1\le j'\le j-1$ have $\ord(\sigma) <\omega\cdot(2j-1)$.

Given $k\ge j$. Let $\widehat T$ be the triangle in \eqref{seqarray} with vertices $\sigma_{j,j}$, $\sigma_{j,k-1}$, $\sigma_{k-1,k-1}$.

\medskip\noindent
\emph{Case 1. $R\not\succ \{\sigma_{j,k}\}$.} Then  
\begin{equation}\label{rough4}
 \sigma_{j,k} \succsim \sigma_m \text{ for some $m\ge 0$.}    
\end{equation}

Suppose the first inequality in \eqref{neword} fails, i.e., $\ord(\sigma_{j,k})>\omega\cdot(2j)$.

Combining this assumption with part \ref{p3}, \eqref{rough}, \eqref{rough4},  and \eqref{orule} applied to $\sigma_*=\sigma_{j,k}$ give 
$$\ord(\{\sigma_0,\ldots,\sigma_{m-1}\}\cup \widehat T)  \supset \{\omega\cdot(2j-1),\omega\cdot(2j-1)+1,\omega\cdot(2j-1)+2,\ldots,\omega\cdot(2j)\}.$$

Since the first set is finite and the second set is infinite, we have a contradiction.
Thus, the first inequality in \eqref{neword} holds.

\medskip\noindent
\emph{Case 2. $R\succ\{\sigma_{j,k}\}$.} One has, for all $m$, $\sigma_m\succ \sigma_{j,k} $. 
Assume the second inequality in \eqref{neword} is false, that is,
\begin{equation}\label{nsj}
\ord(\sigma_{j,k}) > \omega\cdot (2j)+ (k-j)(k-j+1)/2.
\end{equation}

Similar to \eqref{ST}, by using part \ref{p3}, \eqref{rough}, \eqref{nsj},  and \eqref{orule} applied to $\sigma_*=\sigma_{j,k}$ we have
\begin{equation}\label{ST2}
\begin{aligned}
\ord(R\cup \widehat T) \supset
&\{ \omega\cdot(2j-1),\omega\cdot(2j-1)+1,\omega\cdot(2j-1)+2,\ldots,\\
&\quad  \omega\cdot(2j),\omega\cdot(2j)+1,\omega\cdot(2j)+2,\ldots,
\\&\quad  \omega\cdot(2j)+(k-j)(k-j+1)/2\}.    
\end{aligned}
\end{equation}

We proceed as in Case 2 of Step 1 after \eqref{ST} with $\widehat T$ replacing $T$. Then the maximum of $\ord(R\cup \widehat T)$ is 
$$\omega\cdot(2j-1)+\omega+K=\omega\cdot(2j)+K,$$
where integer $K$ is smaller than the cardinality of $\widehat T$. 
This implies $K<(k-1)(k-j+1)/2$, which contradicts \eqref{ST2}. Thus the second inequality of \eqref{neword} holds. This completes the proof of \eqref{neword} for the induction step.

\medskip
\emph{Conclusion.}
By the Induction Principle (in $j$), \eqref{neword} holds for all $j\ge 1$ and $k\ge j$. Hence, we obtain the second inequality  in \eqref{ojk} for all $k\ge j\ge 1$. 
\end{proof}

\begin{lem}\label{concat}
Let $X$ be a countable, totally comparable subset of $\mathcal X$ and $y=(y_n)_{n=1}^\infty$ be a sequence of real numbers.
Then there exists a subsequence $(\varphi(n))_{n=1}^\infty$ of $(n)_{n=1}^\infty$ such that 

{\rm (i)} $y_{\varphi(n)}=0$ for all $n$, or

{\rm (ii)} $y_{\varphi(n)}>0$ for all $n$ and $(X\cup \{y\})_\varphi$ is totally comparable, or

{\rm (iii)} $y_{\varphi(n)}<0$ for all $n$ and $(X\cup \{-y\})_\varphi$ is totally comparable.
\end{lem}
\begin{proof}
Let $I_0$, respectively, $I_+$, $I_-$ be the set of indices $n$ such that $y_n=0$,  respectively, $y_n>0$, $y_n<0$.
If $I_0$ is infinite then (i) holds. If $I_+$ is infinite, then there is a subsequence  $(\psi(n))_{n=1}^\infty$ of $(n)_{n=1}^\infty$ such that $y_{\psi(n)}>0$ for all $n$. Then apply Lemma \ref{total} to $(X\cup\{y\})_\psi$ we obtain (ii).
When $I_-$ is infinite, we apply the argument in (ii) to $X$ and the sequence $(-y)$, hence, obtain (iii).
\end{proof}

\bibliographystyle{plain}
\bibliography{FHJ-main.bib}

\end{document}